\documentclass[11pt,dvipsnames]{article}
\usepackage[utf8]{inputenc}
\usepackage[margin=1in]{geometry}
\usepackage{times} 
\usepackage{tikz}
\usepackage{pgfplots}

\pgfplotsset{width=9cm,compat=1.5}
\pgfplotsset{select coords between index/.style 2 args={
    x filter/.code={
        \ifnum\coordindex<#1\fi
        \ifnum\coordindex>#2\fi
    }
}}
\usepackage{bbm}
\usepackage{amsfonts,amsmath,amssymb,amsthm,diagbox,adjustbox,algorithm,MnSymbol,titling}
\usepackage{booktabs,graphicx,titling,esint,enumitem,fancyhdr,multirow,float,natbib}
\usepackage{hyperref}
\hypersetup{colorlinks,allcolors=blue}
\usepackage{xcolor}
\usepackage{easyReview}

\usepackage{natbib}
\setcitestyle{authoryear}

\usepackage{natbib}

\usepackage{booktabs}
\usepackage{authblk} 
\usepackage{caption}
\usepackage{subcaption}
\usepackage{tikz}
\usetikzlibrary{decorations.text}
\usetikzlibrary{decorations.markings}

\pgfplotsset{
	axis line style={gray}
}
\pgfplotsset{compat=1.10}
\usepgfplotslibrary{fillbetween}
\pgfplotsset{every axis/.append style={
		tick label style={font=\normalsize} }}
\usetikzlibrary{decorations.pathreplacing,angles,quotes}
\definecolor{grey1}{rgb}{0.65,0.65,0.65}
\newtheorem{theorem}{Theorem}[section]
\newtheorem{lemma}[theorem]{Lemma}

\theoremstyle{definition}
\newtheorem{definition}[theorem]{Definition}
\newtheorem{assumption}[theorem]{Assumption}

\newtheorem{remark}[theorem]{Remark}

\newtheorem*{example*}{Example}
\usepackage{algorithm}
\usepackage{setspace}

\newcommand{\bu}{\mathbf{u}}
\newcommand{\bU}{\mathbf{U}}
\newcommand{\bv}{\mathbf{v}}
\newcommand{\cL}{\mathcal{L}}

\makeatletter
\let\old@rule\@rule
\def\@rule[#1]#2#3{\textcolor{Blue}{\old@rule[#1]{#2}{#3}}}
\newcommand*\bigcdot{\mathpalette\bigcdot@{.5}}
\newcommand*\bigcdot@[2]{\mathbin{\vcenter{\hbox{\scalebox{#2}{$\m@th#1\bullet$}}}}}
\makeatother

\DeclareMathOperator*{\argmin}{argmin}

\usepackage{natbib}
\begin{document}
	\title{Open-loop potential difference games with inequality constraints\protect \footnote{This work is supported by Science and Engineering Research Board (SERB), Government of India, Grant no. SERB--EMR/2017/001267.}}
	\author[1]{Aathira Prasad}
	\author[1]{Puduru Viswanadha Reddy}
	\affil[1]{Department of Electrical Engineering, Indian Institute of Technology - Madras, India}
	\affil[ ]{\small \texttt{ee18s033@smail.iitm.ac.in}, \texttt{vishwa@ee.iitm.ac.in}}
	\maketitle

	\begin{abstract}
		Static potential games are non-cooperative games which admit a fictitious function, also referred to as a potential function, such that the minimizers of this function constitute a subset (or a refinement) of the Nash equilibrium strategies of the associated non-cooperative game. In this paper we study a class $N$-player non-zero sum difference games with inequality constraints which admit a potential game structure. In particular, we provide conditions for the existence of an optimal control problem (with inequality constraints) such that the solution of this problem yields an open-loop Nash equilibrium strategy of the corresponding dynamic non-cooperative game (with inequality constraints). Further, we provide a way to construct potential functions associated with this optimal control problem. We specialize our general results to a linear-quadratic setting and provide a linear complementarity problem based approach for computing the refinements of the open-loop Nash equilibria obtained in \cite{Reddy:15}. We illustrate our results with an example inspired by energy storage incentives in a smart grid.
\end{abstract}

	\section{Introduction}
    Multi-agent control systems and related distributed architectures are becoming increasingly popular  with emerging applications such as smart grids, Internet of Things (IoT) systems, intelligent traffic networks and cyber-security. These systems are characterized by the presence of interdependent multiple decision making entities, which are large-scale, distributed, networked and heterogeneous in nature. Game theory has emerged as a powerful tool for analyzing multi-agent systems; see \cite{Manshaei:13}, \cite{Saad:12} and \cite{Zhu:15} for applications of game theory in the areas mentioned above. In particular, dynamic game theory  provides a mathematical framework for modeling multi-agent interactions which evolve over time, and has been successfully used in analyzing a variety of decision problems arising in engineering, economics and management science; see \cite{Basar:99}, \cite{Dockner:00}, \cite{Basar:18}. A significant share of dynamic game models are formulated in an unconstrained setting, that is, where the state and control variables are unconstrained; except for the state equation which captures the evolution of the interaction environment. Constraints appear naturally in real world applications in the form of production-capacity, environmental, market and budget constraints. For example, electric vehicle charging station with limited energy resources impose joint constraints on players. Further, decision problems related to network congestion, which are used in modeling traffic systems, inherently involve capacity constraints. Recently, \cite{Reddy:15} and \cite{Reddy:16} study a class of  non-zero sum difference games with inequality constraints and provide conditions for the existence of Nash equilibria  with open-loop and feedback information structures. 

    The novelty of this paper lies in studying a class of dynamic games with inequality constraints
    which admit a potential game structure. Static potential games were first introduced in \cite{rosenthal:73} and further studied in \cite{monderer:96}. Loosely speaking, a potential game is a non-cooperative game that satisfies the property that a Nash equilibrium of the game is obtained when players jointly optimize a fictitious function, also referred to as a potential function. In other words, a Nash equilibrium of the game is obtained by solving an optimization problem as opposed to a fixed point problem.  An important property associated with  a potential game  is that the strategy profile that provides the optimum for the potential function is a pure strategy Nash equilibrium, and as a result, the existence of pure strategy Nash equilibrium is guaranteed in a potential game. It is well known (\cite{Quint:97}) that a non-cooperative game can admit more than one Nash equilibrium. The multiplicity of equilibria naturally poses a selection problem, and to address this, certain refinements of Nash equilibira have been proposed; see \cite{Myerson:97}. These refinements provide a way of selecting a subset of equilibria based on  additional properties that a Nash equilibrium is required to satisfy. In a potential game, the Nash equilibrium is refined with the property that players are jointly optimizing the potential function in a cooperative manner, even though they are acting strategically optimizing their individual
    objective functions. Due to this property, potential games find applications in the study of network
    congestion games \cite[chapter 18]{Nisan:07}, decentralized learning algorithms (\cite{Marden:09}) and 
    in utility function design (\cite{Marden:18}) for multi-agent systems. Further, in \cite{slade:94}
    this property of potential games has been explored in the context of oligopolistic markets.
    In summary, the static potential games have been studied extensively in literature. 
    
    The objective of this paper is to extend the notion of a potential game in a dynamic setting, and in particular, for dynamic games where inequality constraints appear jointly in control and state variables. Besides providing the conditions on the existence of this class of games, another important objective of our paper is towards computing the refinements of the open-loop Nash equlibria.  We consider a class of finite horizon $N$-player non-zero sum nonlinear difference games with constraint structure similar to \cite{Reddy:15} and \cite{Reddy:16}.     Our contributions are summarized as follows.
   \begin{enumerate}
   	\item  We define the notion of open-loop potential difference game and associate an inequality constrained optimal control problem with the dynamic non-cooperative game with inequality constraints. Further, in Lemma \ref{lem:Difference Condition} and Theorem \ref{thm:verfication} we provide conditions under which the non-cooperative game admits a potential game structure.
   	\item When the potential functions associated with the optimal control problem are not specified, we provide, in Theorem \ref{thm:construction}, a method for constructing the potential functions from the objective functions of the players using the theory of conservative vector fields.   	
   	\item We specialize the obtained results to a linear quadratic setting and characterize in Theorem \ref{th:LQ Restrictions} a class of linear quadratic potential difference games with inequality constraints. Further, in Theorem \ref{thm:LCPOLNE} we provide a linear complementarity problem based method for computing a refinement of the open-loop Nash equilibria.   	
   \end{enumerate}
    
	The paper is organized as follows. In section  \ref{sec:Dynamic Game Model}, we introduce the dynamic game model with the separable structure of the objective function and inequality constraints jointly in state and decision variables. In section \ref{sec: OLDPG}, we define the open-loop dynamic potential games by introducing an optimal control problem associated with the dynamic potential game. Further, we provide the structure of the dynamic game for the existence of  potential functions. We demonstrate the equivalence of the solution of the optimal control problem  as an open-loop Nash equilibrium of the dynamic game. When the potential functions are specified before hand, we also illustrate a procedure for the  construction of potential functions. In section \ref{sec:LQDG}, we specialize these results to the linear quadratic setting. In section \ref{sec:Examples} we illustrate our results with an application motivated by energy storage incentives in a smart grid. Finally, in section \ref{sec:conclusions}  we provide conclusions and future work. 
	
\subsection{Literature review}
The literature on dynamic potential games is limited compared to their static counter part. In a static  {non-cooperative game} the objective functions of the players are required to satisfy certain conditions to admit a potential function. Quite naturally, an extension of these conditions to a dynamic setting imposes structural constraints on the instantaneous and terminal objective functions. In \cite{slade:94}, the author studies oligopolistic markets, and provides conditions under which a single optimization problem provides a Nash equilibrium, both in the static and dynamic settings. The objective function of this optimization problem in this work is referred to as a fictitious function. In particular, in the dynamic case, existence of these functions under the open-loop and feedback information structures has been explored. In \cite{dragone:15}, the authors consider a non-cooperative differential game with open-loop information structure, and  provide conditions for the existence of Hamiltonian potential functions. Potential differential games were studied in \cite{fonseca:18} in an open-loop setting. The authors primarily focused on separable structure of the payoff function from which the potential function could be derived. The unconstrained discrete time game is considered in \cite{gonzalez:14} in a stochastic setting. In \cite{Gonzalez:16}, the authors provide a survey on static and dynamic potential games. All these works study dynamic potential games in discrete and continuous time settings without additional constraints on the state and control variables. In \cite{zazo:16}, the authors consider a discrete time infinite horizon non-cooperative difference games with inequality constraints and provide conditions for the existence of potential functions. They also mention a way for constructing potential functions using the theory of conservative vector fields; an approach followed in static games \cite{monderer:96}. Our work is closer to \cite{zazo:16} in spirit, but differs considerably in the model and approach. The constraint structure in our paper is inspired by our previous works \cite{Reddy:15} and \cite{Reddy:16}. In particular,  we assume that the players have two types of control variables, namely, 1) variables that (directly) affect the dynamics but do not enter the constraints, and 2) variables that enter the constraints jointly with the state variables but do not appear in the dynamics. Further, we assume that the objective functions are separable in the two types of control variables, that is, that there are no cross-terms between them. As was shown in \cite{Reddy:15} and \cite{Reddy:16}, we also demonstrate in this paper that these assumptions are crucial in obtaining a semi-analytical characterizations for Nash equilibria, and in providing a linear complementarity problem based method for computing them. 
More importantly, the linear quadratic dynamic potential games analyzed in this paper provides refinements of the open-loop Nash equilibria obtained in \cite{Reddy:15}, and a procedure for computing them.
	 
 \subsection{Notation} We shall use the following notation. The $n$-dimensional Euclidean space is denoted by $\mathbb R^n$, $n\geq 1$. $A^\prime$ denotes the transpose of a matrix or a vector $A$. $A_1\oplus A_2\oplus \cdots \oplus A_n$ represents the block diagonal matrix obtained by taking the matrices $A_1, A_2,\cdots,A_n$ as diagonal elements in this sequence. The matrix with all entries as zeros is denoted by $\mathbf{0}$,  {all entries as one is denoted by $\mathbf{1}$}, and the identity matrix is represented by $\mathbf{I}$. We call two vectors $x,y\in \mathbb R^n$ complementary if $x\geq 0$, $y\geq0$ and $x^\prime y=0$, and $0\leq x \perp y \geq 0$ denotes this condition. Let $A$ be a $n\times n$ matrix and $a$ be a $n\times 1$ vector. Let $n$ be partitioned as $n=n_1+n_2+\cdots+n_K$. We represent $[A]_{ij}$ as the $n_i\times n_j$ sub-matrix associated with indices $n_i$ (row) and $n_j$ (column), and $[a]_i$ as the $n_i\times 1$ sub-vector associated with indices $n_i$.  $[A]_{i\bullet}$ and $[A]_{\bullet i }$ represent the $i^{th}$ row block matrix of dimensions $n_i \times n $ and $i^{th}$ column matrix  of dimension $n \times n_i$ of the   matrix $A$ respectively. The notation $\mathbf{e}_i$ represents a column vector with $i^{th}$ element as 1 and the rest of the elements as zero.  For differentiation with respect to a vector or matrix variable we follow the convention from \cite{lutkepohl:96}.

	\section{Dynamic game model} \label{sec:Dynamic Game Model}
	In this section, we introduce a class of finite horizon discrete time non-zero sum games with inequality constraints. Let $\mathcal N=\{1,2,\cdots,N\}$ be the set of players and $\mathcal K=\{0,1,2,\cdots,K\}$ be the set of time periods. At each time instant, player $i \in \mathcal N$ chooses the following two types of actions (control variables); variables that enter the dynamics of the system, but not the constraints, denoted by $u^i_k\in U^i_k\subset \mathbb R^{m_i}$; and variables that do not affect the dynamics, but do constrain the decision making process, denoted by $v_k^i\in V_k^i \subset \mathbb R^{s_i}$. Here, $U_k^i$ and $V_k^i$ are the sets of admissible values for the two types of variables. We denote the vector of actions of all players at time period  $k$ by $\mathbf u_k:=\begin{bmatrix} {u^1_k}^\prime&
	{u^2_k}^\prime&\cdots&{u^N_k}^\prime
	\end{bmatrix} ^\prime$ and $\mathbf{v}_k:=\begin{bmatrix} {v^1_k}^\prime&
	{v^2_k}^\prime&\cdots& {v^N_k}^\prime
	\end{bmatrix}^\prime$. 
	The state $ x_k  \in {X}_k \subset \mathbb R^n$ evolves according to the following discrete time dynamics
	\begin{align}
	 {x}_{k+1}=f_k( x_k ,\mathbf{u}_k),~k\in \mathcal K\backslash \{K\},~x_0 \text{  given}.
	\label{eq:statedynamics}
	\end{align}
	Here, $X_k$ denotes the set of admissible state vectors at time period $k$. 	We consider the following joint inequality constraints (also called coupled constraints) associated with players strategies
	\begin{align}
	h_k( x_k ,\mathbf{v}_k)\geq 0,~\mathbf{v}_k\geq 0,~  k\in \mathcal K.
	\label{eq:constraints}
	\end{align}
	Notice, the decision variables $\{u^i_k,~k\in \mathcal K \backslash \{K\},~i\in \mathcal N\}$ do not enter the constraints directly but affect them indirectly through the state variables. 
	Let the strategies of player $i$, that is, the plan of actions for the entire planning period, be denoted by $\tilde{u}^i:=\{u^i_k,~k\in \mathcal K\backslash \{K\}\}$ and 
	$\tilde{v}^i:=\{v^i_k,~k\in \mathcal K \}$. The set of players excluding player $i$ is denoted by $-i:=\mathcal N\backslash \{i\}$. The joint strategy profiles of the players be denoted by $\tilde{\mathbf{u}}:=(\tilde{u}^i,\tilde{u}^{-i})$  and $\tilde{\mathbf{v}}:=(\tilde{v}^i,\tilde{v}^{-i})$, and the corresponding strategy sets are denoted by $\mathcal U$ and $\mathcal V$ respectively. 	Each player $i\in \mathcal  N$ uses her strategies $\tilde{u}^i$ and $\tilde{v}^i$ to minimize the following   objective function given by
	\begin{align}
	 J^{i}(x_0,(\tilde{u}^i,\tilde{u}^{-i}),(\tilde{v}^i,\tilde{v}^{-i}))& =
	g^i_K( x_k ,\mathbf{v}_K) + \sum_{k=0}^{K-1} g^i_k( x_k ,\mathbf{u}_k,\mathbf{v}_k).
	\label{eq:objectives}
	\end{align}
	We have the following assumptions related to the dynamic game \eqref{eq:statedynamics}-\eqref{eq:objectives}.
	\begin{assumption}
		\begin{enumerate}
			\item The admissible action sets $\{U_k^i,~k\in \mathcal K\backslash \{K\},~i\in \mathcal N\}$, are such that the sets of state vectors $\{X_k,~k\in \mathcal K\}$ are convex, and feasible action sets $\{V^i_k( x_k ,v_k^{-i})=\{v_k^i\in \mathbb R^{s_i}~|~ h^i( x_k ,\mathbf{v}_k)\geq 0,~ \mathbf{v}_k\geq 0\},~\forall x_k\in X_k\}$ are non-empty, convex and bounded for all   $k\in \mathcal K$, $i\in \mathcal N$.	A strategy pair $(\tilde{\bu},\tilde{\bv})$ associated with these actions sets is an admissible strategy pair.

			\item The matrices $\left\{\frac{\partial h_k}{\partial v_k^i},~k\in \mathcal K,~i\in \mathcal  N\right \}$ have full rank, so as to satisfy constraint qualification conditions.
			\item The instantaneous cost functions in \eqref{eq:objectives} admit a separable structure, that is,  they do not contain cross terms involving the strategies $\mathbf{\tilde{u}}$  and $\mathbf{\tilde{v}}$ and are represented by			
			 $g_k^i( x_k ,\mathbf{u}_k,\mathbf{v}_k)=
			gu_k^i( x_k ,\mathbf{u}_k)+gv_k^i( x_k ,\mathbf{v}_k)$.
			\item The partial derivatives of the functions $f_k$, $h_k$, and the players' cost functions $\{g_k^i,~k\in \mathcal K,~i\in \mathcal N\}$  exist in their arguments, and are twice continuously differentiable in their arguments. 
		\end{enumerate}
		\label{asm:feasibility}
	\end{assumption}
	\noindent
	Note that in \eqref{eq:objectives} the cost incurred by player $i$ not only depends on his actions but also by the actions of other players $-i$. So, \eqref{eq:statedynamics}-\eqref{eq:objectives} constitutes a dynamic game with inequality constraints, and we refer to it as NZDG from here on. Then, the Nash equilibrium strategies of players, denoted by $(\mathbf{\tilde{u}}^*,\mathbf{\tilde{v}}^*)$, for this class of games is defined as follows.
	\begin{definition}[Nash equilibrium] The strategy profile $(\tilde{\mathbf{u}}^*,\tilde{\mathbf{v}}^*)$ is Nash equilibrium if for every player $i\in \mathcal N$ the strategies $(\tilde{\mathbf{u}}^*,\tilde{\mathbf{v}}^*)$ solves  
		\begin{align}
		\min_{(\tilde{u}^i,\tilde{v}^i)}J^{i}(x_0,(\tilde{u}^i,{\tilde{u}^{-i*}}),(\tilde{v}^i,{\tilde{v}^{-i*}})) \text{ subject to \eqref{eq:statedynamics} and \eqref{eq:constraints}}.
		\end{align}	
		\label{def:NE}
	\end{definition}
	From the above definition the Nash equilibrium strategy is stable against a player's unilateral deviation. In multistage games the interaction environment is dynamic, which is embedded in state variables and their evolution. It is well known  that the Nash equilibrium solution varies with the information used by the players during the (dynamic) decision making process. So, in a dynamic game an information structure must be specified when players design their strategies.  In an open-loop information structure, the players design their strategies using only the knowledge of the time $k$ (and initial state $ {x}_0$). Whereas in a feedback information structure, players design their equilibrium strategies using the knowledge of the state variable. In this paper, we assume open-loop information structure. This implies, that the decision variables entering the dynamics are functions of time. Next, as it is clear from the inequality constraints \eqref{eq:constraints} that the admissible action set of a player depends on the state variable. Therefore player $i$ takes action $v_k^i\in V^i_k( x_k ,v_k^{-i})$ as a function of time $k$, and the feasible set $V^i_k( x_k ,v_k^{-i})$  is parametrized by the state variable and the decision variables of players excluding player $i$; see also \cite{Reddy:15} which considers open-loop information structure for this class of games. 
 \section{Open-loop dynamic potential games}\label{sec: OLDPG}
In this section we introduce the notion of a dynamic potential game and seek to find
conditions under which NZDG is a dynamic potential game. It is well known that the classical potential game is a static concept, (\cite{monderer:96}). A (static) game is said to be a potential game if there exists a potential function, and the minimum of this function provides a pure strategy Nash equilibrium of the game, there by providing refinement (or selection) of Nash equilibria. In other words, when a potential function exists for a game, a Nash equilibrium   can be obtained by solving a minimization problem. So, when we extend this notion of potential game in a dynamic context, it natural to associate an optimal control problem with NZDG. We consider the following
optimal control problem with inequality constraints. 
\begin{subequations} 
\begin{align}
    \mathrm{OCP}:\quad \quad & \min_{\tilde{\mathbf{u}},\tilde{\mathbf{v}}} J( {x}_0,\tilde{\mathbf{u}},\tilde{\mathbf{v}})\label{eq:OCP1}\\
    \text{subject to}~ & {x}_{k+1}=f_{k}( x_k ,\mathbf{u}_k),~ {x}_0 \text{ is given},\label{eq:OCP2}\\
    & h_k( x_k ,\mathbf{v}_k)\geq 0,~\mathbf{v}_k\geq 0, ~ k\in \mathcal K,\label{eq:OCP3}\\
    \text{where}~& J(x_0,\tilde{\mathbf{u}},\tilde{\mathbf{v}})=P_K( x_k ,\mathbf{v}_K)+\sum_{k=0}^{K-1}P_k( x_k ,\mathbf{u}_k,\mathbf{v}_k),\notag 
\end{align}
\label{eq:OCP}
\end{subequations} 
where   $P_k:\mathbb{R}^n \times \mathbb {R}^{m} \times 
\mathbb R^s \rightarrow \mathbb R$ and $P_K:\mathbb{R}^n \times \mathbb {R}^s\rightarrow \mathbb R$ are instantaneous and terminal cost functions, which are continuous and twice continuously differentiable
in their arguments.
\begin{definition}[Open-loop potential difference game] The dynamic game NZDG is referred to as open-loop potential difference game (OLPDG) if there exist cost functions $\{P_k,~k\in \mathcal K\}$ such that the optimal solution of OCP provides an open-loop Nash equilibrium of NZDG. Whenever such cost functions exist, $\{P_k,~k\in \mathcal K\}$ are referred to as potential cost functions.
	\label{def:olpdg}
\end{definition}
In the next two theorems we discuss the necessary and sufficient conditions for 
an admissible pair $(\tilde{\bu}^*,\tilde{\bv}^*) \in \mathcal U \times \mathcal V$ to be an optimal solution of OCP. Towards this end, 
 we define the instantaneous and terminal  Lagrangian functions as follows
\begin{subequations} 
\begin{align}
&\cL_k(x_k,\bu_k,\bv_k,\lambda_{k+1},\mu_k) =P_k(x_k,\bu_k,\bv_k)+\lambda_{k+1}^\prime f_k(x_k,\bu_k)-\mu_k^\prime h_k(x_k,\bv_k),\\
&\cL_K(x_K,\bv_K,\mu_K)=P_K(x_K,v_K)-\mu_K^\prime h_K(x_K,\bv_K).
\end{align}
\label{eq:Lagopt}
\end{subequations} 
\begin{theorem}\label{th:Necessary Cond} 
	Let Assumption \ref{asm:feasibility} holds true.
 	Let $\mathbf{(\tilde{\mathbf{u}}^*, \tilde{\mathbf{v}}^*)}$ be an optimal  admissible pair for OCP, and $\{x_k^*,~k\in \mathcal K\}$ be state trajectory generated by $\tilde{\bu}^*$, the then there exist co-states $\{\lambda_k^*,~k\in \mathcal K\}$  and a  multipliers $\{\mu_k^*,~k\in \mathcal K\}$ such that the following conditions hold true: 
 	\begin{subequations}
 	\begin{align}
 	&\text{for $k\in \mathcal K\backslash \{K\}$}\notag \\
 	\bu_k^*&=\argmin_{\bu_k \in \bU_k}\mathcal L_k(x^*_k,\bu_k,\bv^*_k,\lambda_{k+1}^*,\mu_k^*)\label{eq:nceq1}\\
 	x^*_{k+1}&=\frac{\partial \mathcal L_k}{\partial  \lambda_{k+1}}(x^*_k,\bu^*_k,\bv^*_k,\lambda_{k+1}^*,\mu_k^*),~ x^*_0=x_0 \label{eq:nceq2}\\
 	\lambda^*_k&=\frac{\partial \mathcal L_k}{\partial x_k}(x^*_k,\bu^*_k,\bv^*_k,\lambda_{k+1}^*,\mu_k^*),\label{eq:nceq3a}\\ \lambda^*_K&=\frac{\partial \cL_K}{\partial x_K}(x_K^*,\bv^*_K,\mu_K^*),\label{eq:nceq3b}\\
 	&\text{and for $k\in \mathcal K$}\notag \\
 	& 0\leq h(x_k^*,\bv_k^*)    \perp \mu_k^*	\geq 0 \label{eq:nceq4}\\
 	&0\leq  \frac{\partial \mathcal L_k}{\partial \bv_k} (x^*_k,\bu^*_k,\bv^*_k,\lambda_{k+1}^*,\mu_k^*) \perp \bv^*_k	\geq 0.\label{eq:nceq5}
 	\end{align}
 	\label{eq:nessopt}
 	\end{subequations} 
 \end{theorem}
\begin{proof}The necessary conditions follow directly by applying the discrete-time maximum principle.
\end{proof}
 The following theorem provides conditions under which \eqref{eq:nceq1}-\eqref{eq:nceq5} are also sufficient for optimality of  $\mathbf{(\tilde{\mathbf{u}}^*, \tilde{\mathbf{v}}^*)}$.
\begin{theorem}\label{thm:suff1} Let  {Assumption \ref{asm:feasibility}} holds true.
Let the pair of control strategies $(
\tilde{\bu}^*,\tilde{\bv}^*)\in \mathcal U \times \mathcal V$ and the collection of trajectories $\{x_k^*,\lambda^*_k,\mu^*_k,~k
\in  \mathcal K\}$ satisfy \eqref{eq:nessopt}. Assume that the Lagrangian $\cL_k(x_k,\bu_k,\bv_k,\lambda_{k+1},\mu_k)$ has a minimum with respect to $(\bu_k,\bv_k)$ for all $k\in \mathcal K\backslash \{K\}$. Let the  minimized Lagrangian be given by
 $\cL_k^*(x_k,\lambda_{k+1},\mu_k)=\min_{(\bu_k,\bv_k)}\cL_k(x_k,\bu_k,\bv_k,\lambda_{k+1},\mu_k)$ for $k\in \mathcal K\backslash \{K\}$. Assume that terminal Lagrangian $\cL_K(x_K,\bv_K,\mu_K)$ has a minimum with respect to $\bv_K$, and denote the minimum terminal cost function by $
 \cL^*_K(x_K,\mu_K)=\min_{\bv_K} \cL_K(x_K,\bv_K,\mu_K)$. Then, if $\cL_k^*(x_k,\lambda_{k+1},\mu_k)$ is convex with respect to $x_k$ for all $k\in \mathcal K\backslash \{K\}$ and $\cL^*_K(x_K,\mu_K)$ is convex with respect to $x_K$, then the pair $(\tilde{\bu}^*,\tilde{\bv}^*)$ is optimal for OCP. 
\end{theorem}
\begin{proof}
	For any admissible $(\tilde{\bu},\tilde{\bv})\in \mathcal U\times \mathcal V$ we consider the difference
	\begin{align}
	J(x_0,\tilde{\bu},\tilde{\bv})-J(x_0,\tilde{\bu}^*,\tilde{\bv}^*)=&P_K( {x}_K,\mathbf{v}_K)+\sum_{k=0}^{K-1}P_k( {x}_k,\mathbf{u}_k,\mathbf{v}_k) -P_K( {x}^*_K,\mathbf{v}^*_K)-\sum_{k=0}^{K-1}P_k({x}_k^*,\mathbf{u}^*_k,\mathbf{v}^*_k)\notag\\
	=& \cL_K(x_K,\bv_K,\mu_K^*)-\cL_K(x_K^*,\bv^*_K,\mu_K^*)+{\mu_K^*}^\prime\left(h(x_K,\bv_K)-h(x_K^*,\bv^*_K)\right)\notag \\ +& \sum_{k=0}^{K-1} \cL_k(x_k,\bu_k,\bv_k,\lambda^*_{k+1},\mu^*_k) -\cL_k(x^*_k,\bu^*_k,\bv^*_k,\lambda^*_{k+1},\mu^*_k)\notag \\+&\sum_{k=0}^{K-1}
	-{\lambda_{k+1}^*}^\prime(x_{k+1}-x^*_{k+1})+{\mu^*_k}^\prime\left(h(x_k,\bv_k)-h(x_k^*,\bv_k^*)\right)\notag \\
	\geq &\cL_K^*(x_K,\mu_K^*)-\cL_K^*(x_K^*,\mu_K^*)+ \sum_{k=0}^{K-1}-{\lambda_{k+1}^*}^\prime (x_{k+1}-x^*_{k+1}) \notag \\ +&\sum_{k=0}^{K-1}\cL_k^*(x_k,\lambda_{k+1}^*,\mu_k^*)-\cL^*(x_k^*,\lambda_{k+1}^*,\mu_k^*)\notag\\ +&\sum_{k=0}^K {\mu_k^*}^\prime \left(h(x_k,\bv_k)-h(x_k^*,\bv_k^*)\right).
	\label{eq:objdiff}
	\end{align}
	\noindent
	From the envelope theorem we have $\frac{\partial \cL^*_k}{\partial x_k}(x_k^*,\lambda_{k+1}^*,\mu_k^*)=\frac{\partial \cL_k}{\partial x_k}(x_k^*,\bu_k^*,\bv_k^*,\lambda_{k+1}^*,\mu_k^*)=
	\lambda_k^*$
	and $\frac{\partial \cL^*_K}{\partial x_K}(x_K^*,\mu_K^*)=
	\frac{\partial \cL_K}{\partial x_K}(x_K^*,\bv_K^*,\mu_K^*)=\lambda_K^*$. Then using this and from the convexity of minimized Lagrangian and terminal Lagrangian with respect to state variables we get
	\begin{subequations} 
	\begin{align}
	&\cL_k^*(x_k,\lambda_{k+1}^*,\mu_k^*)-\cL_k^*(x_k^*,\lambda_{k+1}^*,\mu_k^*)\geq \left(\frac{\partial \cL^*_k}{\partial x_k}(x_k^*,{\lambda_{k+1}}^*,\mu_k^*)\right)^\prime (x_k-x_k^*)={\lambda_k^*}^\prime(x_k-x_k^*), \\
		&\cL_K^*(x_K,\mu_K^*)-\cL_K^*(x_K^*,\mu_K^*)\geq \left(\frac{\partial \cL^*_K}{\partial x_K}(x_K^*,\mu_K^*)\right)^\prime (x_K-x_K^*)=
		{\lambda_K^*}^\prime(x_K-x_K^*).
\end{align}
\label{eq:costateenv}
\end{subequations} 
Using \eqref{eq:costateenv} in \eqref{eq:objdiff} we get
\begin{align*}
J(x_0,\tilde{\bu},\tilde{\bv})-J(x_0,\tilde{\bu}^*,\tilde{\bv}^*)&\geq 
\sum_{k=0}^K {\lambda_{k}^*}^\prime(x_{k}-x^*_{k})-\sum_{k=0}^{K-1}{\lambda_{k+1}^*}^\prime (x_{k+1}-x^*_{k+1}) +\sum_{k=0}^K {\mu_k^*}^\prime \left(h(x_k,\bv_k)-h(x_k^*,\bv_k^*)\right)\\
&={\lambda_0^*}^\prime (x_0-x_0^*)+\sum_{k=0}^K {\mu_k^*}^\prime \left(h(x_k,\bv_k)-h(x_k^*,\bv_k^*)\right)\\ &=\sum_{k=0}^K {\mu_k^*}^\prime h(x_k,\bv_k) \geq 0 \hfill.
\end{align*}
The last but one equality follows from  the
complementary condition ${\mu_k^*}^\prime h_k(x_k^*,\bv_k^*)=0$ for $k\in \mathcal K$, and
the initial condition $x_0=x_0^*$. The last inequality follows as multipliers and constraints satisfy
the conditions $\mu_k^*\geq0$ and  $h_k(x_k,\bv_k)\geq 0$ for all $k\in \mathcal K$.
\end{proof} 
\begin{remark}
	 The sufficient condition provided in Theorem \ref{thm:suff1} is an adaptation of 
	 Arrow type sufficient condition \cite[Theorem 3.30]{Grass:08} to a discrete-time setting with inequality constraints, and differs from the nonlinear programming based methods ( \cite{Pearson:66}).
\end{remark}

\subsection{Structure of OLPDG}
\label{sec:structure}
In this subsection we provide conditions under which the optimal solution of OCP provides
an open-loop Nash equilibrium of NZDG. Toward this end, we have the following assumption.
\begin{assumption} 
The cost functions $\{P_k,~k\in \mathcal K\}$ associated with OCP satisfy the following conditions for every $i\in \mathcal N$ 
\begin{subequations} 
\begin{align} 
&\frac{\partial P_k}{\partial u_{k}^{i}}  = \frac{\partial gu_{k}^{i}}{\partial u_{k}^{i}},~ \frac{\partial P_k}{\partial v_{k}^{i}}  = \frac{\partial gv_{k}^{i}}{\partial v_{k}^{i}},~\frac{\partial P_k}{\partial  {x}_{k}}  = \frac{\partial g_{k}^{i}}{\partial  {x}_{k}},~k\in \mathcal K\backslash \{K\},  \label{eq:gradcond1}\\
&\frac{\partial P_K}{\partial  {x}_{K}} =  \frac{\partial g_{K}^{i}}{\partial  {x}_{K}},~
 \frac{\partial P_K}{\partial \mathbf{v}_{K}^i} =  \frac{\partial g_{K}^{i}}{\partial \mathbf{v}_{K}^i}. 
 \label{eq:gradcond2}
\end{align}
\label{eq:gradconditions} 
\end{subequations}  
\label{asm:gradient}     
\end{assumption}
\begin{lemma} \label{lem:Difference Condition} Let Assumption \ref{asm:gradient} holds true.
	Then, the cost functions of the OCP and NZDG satisfy the following relation for each player $i\in \mathcal N$.
\begin{align}
     J(x_0,(\tilde{u}^i,\tilde{u}^{-i}),(\tilde{v}^i,\tilde{v}^{-i})) - J(x_0,(\tilde{w}^i,\tilde{u}^{-i}),(\tilde{z}^i,\tilde{v}^{-i})) &=J^{i}(x_0,(\tilde{u}^i,\tilde{u}^{-i}),(\tilde{v}^i,\tilde{v}^{-i}))\notag \\ &- J^{i}(x_0,(\tilde{w}^i,\tilde{u}^{-i}),(\tilde{z}^i,\tilde{v}^{-i})) 
    \label{eq:Difference Condition}
\end{align} 
$\forall \tilde{u}^i:=\{u^i_k\in U_k^i,~k\in \mathcal{K}\backslash \{K\}\},~\forall\tilde{w}^i:=\{w^i_k\in U_k^i,~k\in \mathcal{K}\backslash \{K\}\},~ \forall \tilde{v}^i:=\{v^i_k \in V_k^i,~k\in \mathcal K \}$ and $\forall \tilde{z}^i:=\{z^i_k \in V_k^i,~k\in \mathcal K \}$.
\end{lemma}
\begin{proof}
	From \eqref{eq:gradcond1} and the separable structure of cost functions as assumed in Assumption \ref{asm:feasibility}, we have
	for every $k\in \mathcal K\backslash \{K\}$
	\begin{subequations} 
		\begin{align} 
		&\frac{\partial }{\partial u_k^i}\left(P_k(x_k,\mathbf{u}_k,\mathbf{v}_k)-g_k^i(x_k,\mathbf{u}_k,\mathbf{v}_k)\right) = \frac{\partial }{\partial u_k^i}\left(P_k(x_k,\mathbf{u}_k,\mathbf{v}_k)-gu_k^i(x_k,\mathbf{u}_k)\right) =  0, \\
		&\frac{\partial }{\partial v_k^i}\left(P_k(x_k,\mathbf{u}_k,\mathbf{v}_k)-g_k^i(x_k,\mathbf{u}_k,\mathbf{v}_k)\right)= \frac{\partial }{\partial v_k^i}\left(P_k(x_k,\mathbf{u}_k,\mathbf{v}_k)-gv_k^i(x_k,\mathbf{v}_k)\right) =  0, \\
		&\frac{\partial }{\partial x_{k}}\left(P_k(x_k,\mathbf{u}_k,\mathbf{v}_k)-g_k^i(x_k,\mathbf{u}_k,\mathbf{v}_k) \right) = 0.  
		\end{align}
		\label{eq:diffequation1}
	\end{subequations} 
	From \eqref{eq:gradcond2} we have 
	\begin{subequations}
	\begin{align}
	&\frac{\partial }{\partial x_{K}}\left(P_K(x_K,\mathbf{v}_K)-g_K^i(x_K,\mathbf{v}_K) \right) = 0 \\ & 
	\frac{\partial }{\partial v_{K}^i}\left(P_K(x_K,\mathbf{v}_K)-g_K^i(x_K,\mathbf{v}_K) \right) = 0.
		\end{align}
		\label{eq:diffequation2}
	\end{subequations}
	Clearly, from \eqref{eq:diffequation1} it follows that the difference 
	function $P_k(x_k,\bu_k,\bv_k)-g_k^i(x_k,\bu_k,\bv_k)$ is independent of (or does not contain)  the variables $x_k$, $u_k^i$ and $v_k^i$. Similarly, from \eqref{eq:diffequation2} it follows that the 
	difference function $P_K(x_K,\bv_k)-g_K^i(x_K,\bv_K)$ does not contain the variables $x_K$ and $v_K^i$. This implies that these difference functions can be expressed as
	\begin{subequations} 
		\begin{align} 
		& P_k( {x}_k,\mathbf{u}_k,\mathbf{v}_k)-g_k^i( {x}_k,\mathbf{u}_k,\mathbf{v}_k) = \Theta^i_k(u_k^{-i},v_{k}^{-i}),~ \forall k \in \mathcal K \backslash \{K\},\\
		& P_K( {x}_K,\mathbf{v}_K)-g_K^i( {x}_K,\mathbf{v}_K) = \Theta^i_K(v_K^{-i}). 
		\end{align} 
		\label{eq:Lemma1 eq3}
	\end{subequations} 
	$\forall u_k^i \in U_k^i$ and $\forall v_k^i \in V_k^i$. Since \eqref{eq:Lemma1 eq3} is satisfied by every $u_k^i \in U_k^i$ and $v_k^i \in V_k^i$, for any $u_k^i,w_k^i \in U_k^i$ and $v_k^i,z_k^i \in V_k^i$ we obtain 
\begin{align*}
P_k(x_k,(u_k^i,u_k^{-i}),(v_k^i,v_k^{-i}))-g_k^i(x_k,(u_k^i,u_k^{-i}),(v_k^i,v_k^{-i})) &=P_k(x_k,(w_k^i,u_k^{-i}),(z_k^i,v_k^{-i}))\\&-g_k^i(x_k,(w_k^i,u_k^{-i}),(z_k^i,v_k^{-i})),\\
P_K(x_K,(v_K^i,v_K^{-i}))-g_K^i(x_K,(v_K^i,v_K^{-i})) &= P_K(x_K,(z_K^i,v_K^{-i}))-g_K^i(x_K,(z_K^i,v_K^{-i})).	
\end{align*}
Upon rearranging the above equations and adding  we get
\begin{subequations} 
	\begin{align} 
 P_k(x_k,(u_k^i,u_k^{-i}),(v_k^i,v_k^{-i}))- P_k(x_k,(w_k^i,u_k^{-i}),(z_k^i,v_k^{-i}) &=g_k^i(x_k,(u_k^i,u_k^{-i}),(v_k^i,v_k^{-i}))\notag\\&-g_k^i(x_k,(w_k^i,u_k^{-i}),(z_k^i,v_k^{-i})),
	\label{eq:instdiff}\\
	P_K(x_K,(v_K^i,v_K^{-i}))-P_K(x_K,(z_K^i,v_K^{-i}))&=g_K^i(x_K,(v_K^i,v_K^{-i}))-g_K^i(x_K,(z_K^i,v_K^{-i})).\label{eq:terminaldiff}
	\end{align}
	\end{subequations} 
	Taking summation of \eqref{eq:instdiff} for all time steps $k \in \mathcal{K}\backslash \{K\}$ and using \eqref{eq:terminaldiff}, we obtain \eqref{eq:Difference Condition}.
\end{proof}
\begin{remark} Lemma \ref{lem:Difference Condition} provides the dynamic counterpart of the principle of exact potential games introduced by Monderer and Shapely \cite[section 2]{monderer:96}.
\end{remark}
\begin{theorem} Let Assumptions \ref{asm:feasibility} and \ref{asm:gradient} hold true. 
	Let the admissible pair $(\tilde{\bu}^*,\tilde{\bv}^*)$ be the optimal solution of OCP. Then
	$(\tilde{\bu}^*,\tilde{\bv}^*)$ is an open-loop Nash equilibrium of NZDG, that is,	NZDG is an OLPDG with potential functions $\{P_k,~k\in \mathcal K\}$.
	\label{thm:verfication}
\end{theorem}
\begin{proof} Let $\{x^*_k,~k\in \mathcal K\}$ be the state trajectory generated by $\tilde{\bu}^*$. As $(\tilde{\bu}^*,\tilde{\bv}^*)$ is optimal for OCP there exist co-states $\{\lambda^*_k,~k\in \mathcal K\}$ and multipliers
$\{\mu^*_k,~k\in \mathcal K\}$ such that the necessary conditions \eqref{eq:nessopt} hold true. Expanding these conditions in terms of cost functions we get
	\begin{subequations} 
\begin{align}
\intertext{for $~k\in \mathcal K\backslash \{K\}$} 
\label{eq:OCP Cond1} & \frac{\partial P_k}{\partial u_{k}^i}(x^*_k,\bu_k^*,\bv_k^*)+\left(\frac{\partial f_k}{\partial u_k^i}(x^*_k,\bu_k^*)\right)^\prime  \lambda_{k+1}^{*} = 0,~i\in \mathcal N, \\ 
	\label{eq:OCP Cond2} & x_{k+1}^{*} = f_k(x^*_k,\bu_k^*),\quad x_0^*=x_0,  \\
	\label{eq:OCP Cond3} & \lambda_k^*=\frac{\partial P_k}{\partial x_{k}}(x^*_k,\bu_k^*,\bv_k^*)+\left(\frac{\partial f_k}{\partial x_k}(x^*_k,\bu_k^*)\right)^{\prime} \lambda_{k+1}^*\notag \\ &\hspace{1.25in}-\left(\frac{\partial h_k}{\partial x_k}(x_k^*,\bv_k^*)\right)^\prime \mu_k^*, \\
	\label{eq:OCP Cond4} & \lambda_{K}^{*}= \frac{\partial P_K}{\partial x_{K}}(x_K^*,\bv_K^*)-\left(\frac{\partial h_K}{\partial x_K}(x^*_K,\bv_K^*)\right)^\prime \mu_{K}^{*}, \\
	\label{eq:OCP Cond5} & 0 \leq \left(\frac{\partial P_k}{\partial v_{k}^i}(x^*_k,\bu_k^*,\bv_k^*)-\left(\frac{\partial h_k}{\partial v_k^i}(x^*_k,\bv_k^*) \right)^\prime \mu_{k}^{*}\right) \perp v_k^{i^{*}} \notag \\&\hspace{1.85in}\geq 0,~~i\in \mathcal N,\\
	\label{eq:OCP Cond6} & 0 \leq \left(\frac{\partial P_K}{\partial v_{K}^i}(x^*_K,\bv_K^*)-\left(\frac{\partial h_K}{\partial v_K^i}(x^*_K,\bv_K^*)\right)^\prime \mu_{K}^{*}\right) \perp v_K^{i^{*}}\notag \\ &\hspace{1.85in} \geq 0,~i\in \mathcal N, \\
  \text{for}& ~k\in \mathcal K, ~
  0\leq h_k(x^*_k,\bv_k^*)\perp \mu_k^* \geq 0. 	\label{eq:OCP Cond7} 
	\end{align}
	\label{eq:optnesscond1}
	\end{subequations}
Next, we write $\bu_k^*=(u_k^{i*},u_k^{-i*})$ and $\bv_k^*=(v_k^{i*},v_k^{-i*})$ and for each player $i\in \mathcal N$ we define the multipliers 
\begin{align}
{\Lambda_k^{i*}}:=\lambda_k^*,k\in \mathcal K\backslash \{K\},~\delta_k^{i*}:=\mu_k^*,~k\in \mathcal K.
\label{eq:samemultipliers}
\end{align}
Now from Assumption \ref{asm:gradient} and using the above notation we write \eqref{eq:optnesscond1}
as follows.
\begin{subequations}  
	\begin{align}
\text{for} &~k\in \mathcal K\backslash \{K\} \notag\\ 
&\frac{\partial g_k^i(x^*_{k},(u_{k}^{i},u_{k}^{-i*}),(v_{k}^{i},v_{k}^{-i*}))}{\partial u_k^i}\Big|_{u_k^i = u_k^{i*}}+\left(\frac{\partial f_k(x^*_{k},(u_{k}^{i},u_{k}^{-i*}))}{\partial u_k^i}\Big|_{u_k^i = u_k^{i*}}\right)^{\prime}\Lambda_{k+1}^{i*} = 0, 	\label{eq:OLNEcond1} \\
& x_{k+1}^* = f(x_{k}^{*},\mathbf{u}_{k}^{*}),~x_0^*=x_0,  	 \label{eq:OLNEcond2} \\
 \Lambda_k^{i*} &=  \frac{\partial g_k^i(x_{k},(u_{k}^{i*},u_{k}^{-i*}),(v_{k}^{i*},v_{k}^{-i*}))}{\partial x_k}\Big|_{x_k = x_k^{*}}+\left(\frac{\partial f_k(x_{k},(u_{k}^{i*},u_{k}^{-i*}))}{\partial x_k}\Big|_{x_k = x_k^{*}}\right)^{\prime}\Lambda_{k+1}^{i*}\notag \\ 
&\qquad-\left(\frac{\partial h_k(x_{k},(v_{k}^{i*},v_{k}^{-i*}))}{\partial x_k}\Big|_{x_k = x_k^{*}}\right)^{\prime}\delta_{k}^{*},\\
 \Lambda_K^{i*}& =  \frac{\partial g_K^i(x_{K},(v_{K}^{i*},v_{K}^{-i*}))}{\partial x_K}\Big|_{x_K = x_K^{*}} -\left(\frac{\partial h_K(x_{K},(v_{K}^{i*},v_{K}^{-i*}))}{\partial x_K}\Big|_{x_K = x_K^{*}}\right)^{\prime}\delta_K^{i*},\\
&0 \leq \frac{g_k^i(x^*_{k},(v_{k}^{i},v_{k}^{-i*}))}{\partial v_k^i}\Big|_{v_k^i =v_k^{i*}}\notag -\left(\frac{\partial h_k(x^*_{k},(v_{k}^{i},v_{k}^{-i*}))}{\partial v_k^i}\Big|_{v_k^i = v_k^{*}}\right)^{\prime}\delta_k^{i*}  \perp v_k^{i*} \geq 0,\\
& 0 \leq \frac{g_K^i(x^*_{K},(v_{K}^{i},v_{K}^{-i*}))}{\partial v_K^i}\Big|_{v_K^i =v_K^{i*}}-\left(\frac{\partial h_K(x^*_{K},(v_{K}^{i},v_{K}^{-i*}))}{\partial v_K^i}\Big|_{v_K^i = v_K^{*}}\right)^{\prime}\delta_k^{i*}  \perp v_K^{i*}  \geq 0,\\
  \text{for} &~k\in \mathcal K, ~ 0 \leq h_{k}(x_{k}^*,\mathbf{v}_{k}^*) \perp \delta_k^{i*} \geq 0. 
	\end{align}
	\label{eq:nessonle1}
\end{subequations} 

Next, we consider Player $i$'s optimal control problem 
\eqref{eq:constraints} when players in $-i=\mathcal N\backslash \{i\}$ use strategies $(\tilde{u}^{-i*},\tilde{v}^{-i*})$. Taking the co-state vectors as
$\{\Lambda_k^i,~k\in \mathcal K\}$ and Lagrange multipliers as $\{\delta_k^i,~k\in \mathcal K\}$ we write the instantaneous Lagrangian and terminal Lagrangian functions associated with Player $i$'s problem are given as
\begin{subequations} 
	\begin{align}
	&\mathcal L _k^i(x_k,(u_k^i,u_k^{-i*}),(v_k^{i},v_k^{-i*}),\Lambda_{k+1}^{i},\delta_k^{i})=g^i_k(x_k,(u_k^i,u_k^{-i*}),(v_k^{i},v_k^{-i*})) +{\Lambda^i_{k+1}}^\prime f_k(x_k,(u_k^i,u_k^{-i*}))\notag \\&\hspace{4.4in}-{\delta^i_k}^\prime h_k(x_k,(v_k^{i},v_k^{-i*})),\\
&	\cL^i_K(x_K,(v_K^{i},v_K^{-i*}),\mu_K)=g^i_K(x_K,(v_k^{i},v_K^{-i*}))-{\delta^i_K}^\prime h_K(x_K,(v_K^{i},v_K^{-i*})).
	\end{align}
	\label{eq:LagNE}
\end{subequations} 
Representing the equations \eqref{eq:nessonle1} in  terms of the instantaneous Lagrangian and terminal Lagrangian functions \eqref{eq:LagNE} 
associated with Player $i$'s optimal control problem \eqref{eq:constraints} we get  
\begin{subequations}
	\begin{align}
	\text{for } &k\in \mathcal K\backslash \{K\}\notag \\
	&u_k^{i*}=\argmin_{u^i_k\in U^i_k}\mathcal L^i_k(x^*_k,(u_k^i,u_k^{-i*}),(v_k^{i},v_k^{-i*}),\Lambda_{k+1}^{i*},\delta_k^{i*}),  \\
	&x^*_{k+1}=\frac{\partial \cL^i_k}{\partial  \Lambda^{i*}_{k+1}}(x^*_k,(u_k^{i*},u_k^{-i*}),(v_k^{i*},v_k^{-i*}),\Lambda_{k+1}^{i*},\delta_k^{i*}),\quad x^*_0=x_0,\\
	&\lambda^*_k=\frac{\partial \mathcal L^i_k}{\partial x_k}(x^*_k,(u_k^{i*},u_k^{-i*}),(v_k^{i*},v_k^{-i*}),\Lambda_{k+1}^{i*},\delta_k^{i*}),\\& \lambda^*_K=\frac{\partial \cL^i_K}{\partial x_K}(x_K^*,(v_k^{i*},v_k^{-i*}),\delta_K^{i*}),\\
	&0\leq  \frac{\partial \mathcal L^i_k}{\partial v^i_k} (x^*_k,(u_k^{i},u_k^{-i*}),(v_k^{i},v_k^{-i*}),\lambda_{k+1}^*,\mu_k^*) \perp v^{i*}_k	\geq 0,\\
		&0\leq  \frac{\partial \cL^i_K}{\partial v^i_K} (x^*_K,(v_K^{i},v_K^{-i*}),\Lambda_{k+1}^{i*},\delta_K^{i*}) \perp v^{i*}_K	\geq 0,\\
		\text{for } &k\in \mathcal K,~ 0\leq h(x_k^*,(v_k^{i},v_k^{-i*}))    \perp \delta_k^{i*}	\geq 0.
	\end{align}
	\label{eq:nessonle2}
\end{subequations} 
Clearly, \eqref{eq:nessonle2} constitute the necessary conditions for optimality of  associated with 
Player i's optimal control problem \eqref{eq:constraints} where strategies of players in $-i$ are fixed at
$(\tilde{u}^{-i*},\tilde{v}^{-i*})$. In other words, $(\tilde{u}^{i*},\tilde{v}^{i*})$ is a candidate best response to $(\tilde{u}^{-i*},\tilde{v}^{-i*})$. 
Next,   we  show that the strategy $(\tilde{u}^{i*},\tilde{v}^{i*})$  indeed minimizes the objective  $J^i(x_0,(\tilde{u}^{i^*},\tilde{u}^{-i*}),(\tilde{v}^{i^*},\tilde{v}^{-i*}))$ subject to 
the dynamics $x_{k+1}=f_k(x_k,(u_k^i,u_k^{-i*}))$, $k\in \mathcal K \backslash \{K\}$ and constraints $h_k(x_k,(v_k^i,v_k^{-i*}))\geq 0$, $v_k^i\geq 0$, $k\in \mathcal K$. {Since $(\mathbf{\tilde{u}^*,\tilde{v}^*})$ is an optimal solution of the OCP,  the following inequality holds true 
\begin{align}
J(x_0,(\tilde{u}^{i^*},\tilde{u}^{-i*}),(\tilde{v}^{i^*},\tilde{v}^{-i*})) \leq J(x_0,(\tilde{u}^{i},\tilde{u}^{-i*}),(\tilde{v}^{i},\tilde{v}^{-i*})),~\forall (\tilde{u}^i,\tilde{v}^i) 
\label{eq:OCP max}
    \end{align}
 with $\tilde{v}_k^i$ satisfying the constraint $h_k( x_k ,({v}_k^i,{v}_k^{\mbox{-}i^*}))\geq 0 ~\forall k \in \mathcal{K}$, where $ x_k $ is the state trajectory evolved according to the action set $(\tilde{u}^i,\tilde{u}^{\mbox{-}i^*})$.} From lemma \ref{lem:Difference Condition}, we have
\begin{multline*}
 J(x_0,(\tilde{u}^{i^*},\tilde{u}^{-i*}),(\tilde{v}^{i^*},\tilde{v}^{-i*})) -     J(x_0,(\tilde{u}^{i},\tilde{u}^{-i*}),(\tilde{v}^{i},\tilde{v}^{-i*}))\\ =  J^i(x_0,(\tilde{u}^{i^*},\tilde{u}^{-i*}),(\tilde{v}^{i^*},\tilde{v}^{-i*}))-     J^i(x_0,(\tilde{u}^{i},\tilde{u}^{-i*}),(\tilde{v}^{i},\tilde{v}^{-i*})).
  \end{multline*}
 Using the above observation in \eqref{eq:OCP max} we get
  \begin{align}
    J^i(x_0,(\tilde{u}^{i^*},\tilde{u}^{-i*}),(\tilde{v}^{i^*},\tilde{v}^{-i*})) \leq J^i(x_0,(\tilde{u}^{i},\tilde{u}^{-i*}),(\tilde{v}^{i},\tilde{v}^{-i*})) \quad \forall ((\tilde{u}^{i},\tilde{u}^{-i*}),(\tilde{v}^{i},\tilde{v}^{-i*}))
    \end{align} 
 with $\tilde{v}_k^i$ satisfying the constraint $h_k( x_k ,({v}_k^i,{v}_k^{\mbox{-}i^*}))\geq 0 ~\forall k \in \mathcal{K}$.  This implies  $(\tilde{u}^{i^*},\tilde{v}^{i^*})$ is indeed a best response to $(\tilde{u}^{-i*},\tilde{v}^{-i*})$. So, $(\mathbf{\tilde{u}^*,\tilde{v}^*})$ is an open-loop Nash equilibrium of NZDG. Following Definition \ref{def:olpdg} we have NZDG is an OLPDG with potential  functions 
$\{P_k,~k\in \mathcal K\}$.
\end{proof}
\begin{remark} The Nash equilibira correspond to fixed points of best response mapping defined over
	joint strategy sets of players, and as a result there can exist more than one equilibria in a non-comparative game. When OCP associated with OLPDG has an optimal solution, then from Theorem \ref{thm:verfication}   this solution is an open-loop Nash equilibrium of NZDG, thereby providing a refinement of the open-loop Nash equilibria. This implies that, solution of the OCP provides a way for selecting one among possibly many open-loop Nash equiliria.
	\end{remark} 
	
	\begin{remark}  Notice, the constraints in \eqref{eq:constraints} are coupled. Rosen in \cite{Rosen:65} studied non-cooperative games with couple constraints. The equilibria in these games are referred to as normalized Nash equlibria, and are characterized by the property that
		multiplier vector associated with constraints in each player's individual optimization problem is co-linear with a common multiplier vector.  In our work, we observe a similar feature by construction, that is, in \eqref{eq:samemultipliers} the obtained open-loop Nash equilibrium has the property that the associated multipliers and co-state variables are same for all players. 
\end{remark}
\subsection{Construction of potential cost functions} 
In section \ref{sec:structure}, a sufficient condition is provided to verify if NZDG is an OLPDG given the potential cost functions $\{P_k,~k\in \mathcal K\}$ of the OCP. However, in practice these functions are not available before hand. In these settings, it is desirable to construct potential functions using players' cost functions. This construction procedure involves two steps. First, to verify if   NZDG is an OLPDG, and then to construct the potential functions. Toward this end, we recall the following necessary and sufficient condition existence of conservative vector fields from multi-variable calculus (\cite{apostol:69}).
\begin{lemma}[Conservative vector field] Let $\Omega$ be a convex open subset in $\mathbb R^n$. Let $F:\Omega \rightarrow \mathbb R^n$ be a vector field  with continuous derivatives defined over $\Omega$.  The following conditions on $F$ are equivalent
	\begin{enumerate}
		\item[a.]There exists a scalar potential function $\Pi:\Omega\rightarrow \mathbb R$ such that
		$F(\omega)=\nabla \Pi(\omega)$ for all $\omega \in \Omega$, where $\nabla$ denotes the gradient operator.
		\item[b.] The partial derivatives satisfy 
		\begin{align}
		\frac{\partial [F(\omega)]_i}{\partial [\omega]_j} = \frac{\partial [F(\omega)]_j}{\partial [\omega]_i},~\forall \omega \in \Omega,~i,j=1,2,\cdots,n.
		\label{eq:symcond}
		\end{align} 
		\item[c.] Let $a$ be a fixed point in $\Omega$, and $\mathcal{C}\subset \Omega$ be a piecewise smooth curve
		joining $a$ with an arbitrary point $\omega \in \Omega$. Then the potential function $\Pi$ satisfies 
		\begin{align}
		\Pi(\omega) -\Pi(a)=\int_\mathcal{C} F(\omega)\bigcdot d\omega = \int_0^1 F(\alpha(z)) \bigcdot \frac{\partial \alpha}{\partial z} ~dz,
		\end{align}
		where $\bigcdot$ is the dot product, and $\alpha: [0, 1] \rightarrow \mathcal{C}$ is a bijective parametrization of the curve $\mathcal{C}$ such that $\alpha(0)=a$ and $\alpha(1)=\omega$.
\end{enumerate}
A vector field  $F:\Omega\rightarrow \mathbb R^n$ satisfying these conditions is called a conservative vector field.
\label{lem:potential}
\end{lemma} 
\begin{proof} See \cite[Theorems 10.4, 10.5 and 10.9]{apostol:69}.
	\end{proof} 
\begin{remark}
	A consequence of the condition \eqref{eq:symcond} is that the Jacobian matrix of the vector field $F:\Omega \rightarrow \mathbb R^n$ evaluated at   $\omega \in \Omega $ is symmetric for all $\omega\in \Omega$. 
\end{remark}
\begin{theorem}	Let Assumption \ref{asm:feasibility}.(1) holds true. Let the players' utility functions satisfy the following conditions, for all $i,j\in \mathcal N$, 
	\begin{subequations}\label{eq:Potential NC}
		\begin{align}  
		& \frac{\partial^2 g_k^i}{\partial (u_k^j)^\prime \partial u_k^i} =\left( \frac{\partial^2 g_k^j}{\partial (u_k^i)^\prime \partial u_k^j}\right)^\prime,~k\in \mathcal K\backslash \{K\}, \label{eq:ux}\\
		&\frac{\partial g_k^i}{\partial x_k} = \frac{\partial g_k^j}{\partial x_k} \label{eq: partial x},~k\in \mathcal K,\\
		& \frac{\partial^2 g_k^i}{\partial (v_k^j)^\prime  \partial v_k^i} = \left(\frac{\partial^2 g_k^j}{\partial (v_k^i)^\prime  \partial v_k^j}\right)^\prime,~k\in \mathcal K, \label{eq:vxx}
		\end{align}
	\end{subequations}
	then NZDG is a OLPDG. 
	Let the vector fields $F_k({x}_k,\mathbf{u}_k,\mathbf{v}_k)$ at $(x_k,\bu_k,\bv_k)\in {\Omega}_k$, where ${\Omega}_k := {X}_k \times \prod_{i \in \mathcal{N}}U_k^i \times \prod_{i \in \mathcal{N}}V_k^i$ of dimension $(n+m+s) \times 1$ and $F_K({x}_K,\mathbf{v}_K)$ at $(x_K,\bv_K)\in {\Omega}_K$ where  $\Omega_K:=  {X}_K \times \prod_{i \in \mathcal{N}}V_K^i$ of dimension $(n+s) \times 1$ be defined by 
	\begin{align*}
	 &F_k({x}_k,\bu_k,\bv_k) = \begin{bmatrix}\frac{\partial g_k^1}{\partial {x}_k}^\prime & \frac{\partial g_k^1}{\partial u_k^1}^\prime & \cdots & \frac{\partial g_k^N}{\partial u_k^N}^\prime & \frac{\partial g_k^1}{\partial v_k^1}^\prime& \cdots & \frac{\partial g_k^N}{\partial v_k^N}^\prime  \end{bmatrix}^\prime,~ k \in \mathcal{K}\backslash \{K\}, \\
	 &F_K(x_K,\bv_K) = \begin{bmatrix}\frac{\partial g_K^1}{\partial {x}_k}^\prime & \frac{\partial g_K^1}{\partial v_K^1}^\prime& \cdots & \frac{\partial g_K^N}{\partial v_K^N}^\prime   \end{bmatrix}^\prime,
	\end{align*}
	then the  instantaneous and terminal potential  functions are given by 
	\begin{subequations}
	\begin{align} 
	&P_k(x_k,\bu_k,\bv_k) =c_k+\int_0^1 F_k(\alpha_k(z)) \bigcdot  \frac{\partial \alpha_k(z)}{\partial z}dz,~k\in \mathcal K, \label{eq:Pfk}\\
	&P_K(x_K,\bv_K) =c_K+\int_0^{1} F_K(\alpha_K(z)) \bigcdot  \frac{\partial \alpha_K(z)}{\partial z}dz,
	\label{eq:PfK}
	\end{align}
    \end{subequations}
 where $\alpha_k:[0,1]\rightarrow \mathcal{C}_k$ ($k\in \mathcal K$) is a bijective parametrization 
of a piece-wise smooth curve $\mathcal{C}_k\subset  {\Omega}_k$ such that 
$\alpha_k(0)=  (x_{0k},\bu_{0k},\bv_{0k})$, $\alpha_k(1)=(x_{k},\bu_{k},\bv_{k})$ for $k\in\mathcal K \backslash \{K\}$,  $\alpha_K(0)=  (x_{0K},\bv_{0K})$, $\alpha_K(1)=(x_{K},\bv_{K})$, and $\{c_k,k\in \mathcal K\}$ are constants.
\label{thm:construction}
\end{theorem}
\begin{proof}  
	We write the vector field $F_k,~k \in \mathcal{K} \backslash K $ as $F_k = \begin{bmatrix}{F_k^x}^\prime & {F_k^u}^\prime &  {F_k^v}^\prime \end{bmatrix}^\prime_{(n+m+s) \times 1} $ such that
	\begin{align*}
	& F_k^x = \left[\frac{\partial g_k^1}{\partial x_k }\right]_{n \times 1},
	~ F_k^u =\begin{bmatrix} \frac{\partial g_k^1}{\partial u_k^1}^\prime & \frac{\partial g_k^2}{\partial u_k^2}^\prime & \hdots & \frac{\partial g_k^N}{\partial u_k^N}^\prime
	\end{bmatrix}^\prime_{m \times 1},\\& F_k^v = \begin{bmatrix} \frac{\partial g_k^1}{\partial v_k^1}^\prime & \frac{\partial g_k^2}{\partial v_k^2}^\prime & \hdots & \frac{\partial g_k^N}{\partial v_k^N}^\prime
	\end{bmatrix}^\prime_{s \times 1}.
	\end{align*}
	Let $\mathbf{w}_k,~k \in \mathcal{K} \backslash \{K\} $ be the $(n+m+s) \times 1$ vector given by $\mathbf{w}_k = \begin{bmatrix}
	x_k^\prime & \bu_k^\prime &  \bv_k^\prime
	\end{bmatrix}^\prime_{(n+m+s) \times 1}$. Therefore, we can write the Jacobian matrix, $ \mathcal{J}_k,~k \in \mathcal{K} \backslash\{k\}$  as 
	\begin{align}
	\label{eq:Jacobiank}
	\mathcal{J}_k = \frac{\partial F_k}{\partial (\mathbf{w}_k)^\prime  } & =  \begin{bmatrix}\frac{\partial F_k^x}{\partial (x_k)^\prime} & \frac{\partial F_k^x}{\partial (\bu_k)^\prime} & \frac{\partial F_k^x}{\partial (\bv_k)^\prime} \\[0.2cm]
	\frac{\partial F_k^u}{\partial (x_k)^\prime} & \frac{\partial F_k^u}{\partial (\bu_k)^\prime} & \frac{\partial F_k^u}{\partial (\bv_k)^\prime} \\[0.2cm] 
	\frac{\partial F_k^v}{\partial (x_k)^\prime} & \frac{\partial F_k^v}{\partial (\bu_k)^\prime} & \frac{\partial F_k^v}{\partial (\bv_k)^\prime}
	\end{bmatrix}_{(n+m+s) \times (n+m+s)},
	\end{align}
	where 
	\begin{align*}
	&\frac{\partial F_k^x}{\partial (x_k)^\prime } = \left[ \frac{\partial^2 g_k^1}{\partial (x_k)^\prime \partial x_k}\right]_{n \times n },~\frac{\partial F_k^u }{\partial (x_k)^\prime } = \begin{bmatrix} \frac{\partial^2 g_k^1}{\partial (x_k)^\prime \partial u_k^1}  \\[0.2cm]  \frac{\partial^2 g_k^2}{\partial (x_k)^\prime \partial u_k^2} \\ \vdots \\ \frac{\partial^2 g_k^N}{\partial (x_k)^\prime \partial u_k^N} \end{bmatrix}_{m \times n }, ~\frac{\partial F_k^v }{\partial (x_k)^\prime } = \begin{bmatrix} \frac{\partial^2 g_k^1}{\partial (x_k)^\prime \partial v_k^1}  \\[0.2cm] \frac{\partial^2 g_k^2}{\partial (x_k)^\prime \partial v_k^2} \\ \vdots \\ \frac{\partial^2 g_k^N}{\partial (x_k)^\prime \partial v_k^N} \end{bmatrix}_{s \times n },  
	\end{align*}
	\begin{align*}
	& \frac{\partial F_k^x}{\partial (\bu_k)^\prime } = \begin{bmatrix}
	\frac{\partial^2 g_k^1}{\partial (u_k^1)^\prime \partial x_k }& \frac{\partial^2 g_k^1}{\partial (u_k^2)^\prime \partial x_k } & \hdots & \frac{\partial^2 g_k^1}{\partial (u_k^N)^\prime \partial x_k }
	\end{bmatrix}_{n \times m},~\frac{\partial F_k^x}{\partial (\bv_k)^\prime  } = \begin{bmatrix}
	\frac{\partial^2 g_k^1}{\partial (v_k^1)^\prime \partial x_k }& \frac{\partial^2 g_k^1}{\partial (v_k^2)^\prime \partial x_k } & \hdots & \frac{\partial^2 g_k^1}{\partial (v_k^N)^\prime \partial x_k }\end{bmatrix}_{n \times s}, 
	\end{align*}
	\begin{align*}
	&  \frac{\partial F_k^u}{\partial (\bu_k)^\prime  } = \begin{bmatrix} \frac{\partial^2 g_k^1}{\partial (u_k^1)^\prime \partial u_k^{1}} & \frac{\partial^2 g_k^1}{\partial (u_k^2)^\prime \partial u_k^{1}}& \hdots & \frac{\partial^2 g_k^1}{\partial (u_k^N)^\prime \partial u_k^{1}} \\[0.2cm]
	\frac{\partial^2 g_k^2}{\partial (u_k^{1})^\prime \partial u_k^2} & \frac{\partial^2 g_k^2}{\partial (u_k^2)^\prime \partial u_k^{2}}& \hdots & \frac{\partial^2 g_k^2}{\partial (u_k^N)^\prime \partial u_k^{2}} \\
	\vdots & \vdots & \ddots & \vdots \\
	\frac{\partial^2 g_k^N}{\partial (u_k^{1})^\prime \partial u_k^N} & \frac{\partial^2 g_k^N}{\partial (u_k^{2})^\prime  \partial u_k^N}& \hdots & \frac{\partial^2 g_k^N}{\partial (u_k^{N})^\prime \partial u_K^N}
	\end{bmatrix}_{m \times m},~\frac{\partial F_k^v}{\partial (\bu_k)^\prime} = \begin{bmatrix} \frac{\partial^2 g_k^1}{\partial (u_k^{1})^\prime \partial v_k^1 }& \frac{\partial^2 g_k^1}{\partial (u_k^2)^\prime \partial v_k^{1}}& \hdots & \frac{\partial^2 g_k^1}{\partial (u_k^N)^\prime \partial v_k^{1}} \\[0.2cm]
	\frac{\partial^2 g_k^2}{\partial (u_k^{1})^\prime \partial v_k^2} & \frac{\partial^2 g_k^2}{\partial (u_k^{2})^\prime \partial v_k^2 }& \hdots & \frac{\partial^2 g_k^2}{\partial (u_k^N)^\prime \partial v_k^{2}} \\
	\vdots & \vdots & \ddots & \vdots \\
	\frac{\partial^2 g_k^N}{\partial (u_k^{1})^\prime \partial v_k^N} & \frac{\partial^2 g_k^N}{\partial (u_k^{2})^\prime \partial v_k^N}& \hdots & \frac{\partial^2 g_k^N}{\partial (u_k^{N})^\prime \partial v_k^N}
	\end{bmatrix}_{s \times m},
	\end{align*}
	\begin{align*} 
	&\frac{\partial F_k^u}{\partial (\bv_k)^\prime } = \begin{bmatrix} \frac{\partial^2 g_k^1}{\partial (v_k^{1})^\prime \partial u_k^1 }& \frac{\partial^2 g_k^1}{\partial (v_k^2)^\prime \partial u_k^{1}}& \hdots & \frac{\partial^2 g_k^1}{\partial (v_k^N)^\prime \partial u_k^{1}} \\[0.2cm]
	\frac{\partial^2 g_k^2}{\partial (v_k^{1})^\prime \partial u_k^2} & \frac{\partial^2 g_k^2}{\partial (v_k^{2})^\prime \partial u_k^2 }& \hdots & \frac{\partial^2 g_k^2}{\partial (v_k^N)^\prime \partial u_k^{2}} \\
	\vdots & \vdots & \ddots & \vdots \\
	\frac{\partial^2 g_k^N}{\partial (v_k^{1})^\prime\partial u_k^N} & \frac{\partial^2 g_k^N}{\partial (v_k^{2})^\prime \partial u_k^N}& \hdots & \frac{\partial^2 g_k^N}{\partial (v_k^{N})^\prime \partial u_k^N}
	\end{bmatrix}_{m \times s},~\frac{\partial F_k^v}{\partial (\bv_k)^\prime } = \begin{bmatrix} \frac{\partial^2 g_k^1}{\partial (v_k^1)^\prime \partial v_k^{1}} & \frac{\partial^2 g_k^1}{\partial (v_k^2)^\prime \partial v_k^{1}}& \hdots & \frac{\partial^2 g_k^1}{\partial (v_k^N)^\prime \partial v_k^{1}} \\
	\frac{\partial^2 g_k^2}{\partial (v_k^{1})^\prime \partial v_k^2} & \frac{\partial^2 g_k^2}{\partial (v_k^2)^\prime \partial v_k^{2}}& \hdots & \frac{\partial^2 g_k^2}{\partial (v_k^N)^\prime \partial v_k^{2}} \\[0.2cm]
	\vdots & \vdots & \ddots & \vdots \\
	\frac{\partial^2 g_k^N}{\partial (v_k^{1})^\prime \partial v_k^N} & \frac{\partial^2 g_k^N}{\partial (v_k^{2})^\prime \partial v_k^N}& \hdots & \frac{\partial^2 g_k^N}{\partial (v_k^N)^\prime \partial v_k^{N}}
	\end{bmatrix}_{s \times s}.
	\end{align*} 
	Firstly, as $\frac{\partial F_k^x}{\partial (x_k)^\prime } = \frac{\partial^2 g_k^1}{\partial (x_k)^\prime \partial x_k}$ we have that $\frac{\partial F_k^x}{\partial (x_k)^\prime }$ is a symmetric matrix. 
	The off-diagonal block matrices in  $\frac{\partial F_u}{\partial (\bu_k)^\prime}$ satisfy \eqref{eq:ux}
	and this implies $\frac{\partial F_u}{\partial (\bu_k)^\prime}$ is symmetric matrix. Similarly, from 
	\eqref{eq:vxx}, we have that  $\frac{\partial F_v}{\partial (\bv_k)^\prime}$ is also a symmetric matrix. 
	Next, from \eqref{eq: partial x}, $ \frac{\partial g_k^i}{\partial x_k} = \frac{\partial g_k^j}{\partial x_k}$, and as cost functions are twice continuously differentiable, we have $\frac{\partial^2 g_k^i}{\partial (u_k^j)^\prime  \partial x_k} = \frac{\partial^2 g_k^j}{\partial (u_k^j)^\prime \partial x_k}$.
	Next, from the symmetry property of mixed partials we get $\frac{\partial^2 g_k^j}{\partial (u_k^j)^\prime  \partial x_k} = \left(\frac{\partial^2 g_k^j}{\partial (x_k)^ \prime \partial u_k^j}\right)^\prime $, and then using \eqref{eq: partial x} we have
	$ \frac{\partial^2 g_k^i}{\partial (u_k^j)^\prime \partial x_k}  = \left( \frac{\partial^2 g_k^j}{\partial (x_k)^\prime  \partial u_k^j}\right)^\prime$, which implies $\frac{\partial F_k^x}{\partial (\bu_k)^\prime} = \left(\frac{\partial F_k^u}{\partial (x_k)^\prime} \right) ^\prime $. Again, using similar arguments we can show that 
	$\frac{\partial F_k^x}{\partial (\bv_k)^\prime } = \left(\frac{\partial F_k^v}{\partial (x_k)^\prime } \right) ^\prime$. From the separable structure of the cost functions we have
	$\frac{\partial^2 g_k^i}{\partial (v_k^j)^\prime  \partial u_k^i} =\left( \frac{\partial^2 g_k^j}{\partial (u_k^i)^\prime  \partial v_k^j}\right)^\prime = \mathbf{0}$, which implies $\frac{\partial F_k^u}{\partial (\bv_k)^\prime } = \left(\frac{\partial F_k^v}{\partial (\bu_k)^\prime } \right) ^\prime = \mathbf{0}$. Clearly, from these observations we see that Jacobian matrix \eqref{eq:Jacobiank} is a symmetric matrix, and as a result $F_k$ is a conservative vector field for all $k\in \mathcal K \backslash \{K\}$.
	At the terminal instant we write $F_K = \begin{bmatrix}F_K^x & F_K^v \end{bmatrix}_{(n+s) \times 1}$ and define $\mathbf{w}_K = \begin{bmatrix} 	x_K^\prime &  \bv_K^\prime \end{bmatrix}^\prime_{(n+s) \times 1}$. From \eqref{eq:Potential NC} and using the same reasoning as above we can show that the Jacobian $ \mathcal{J}_K =  \frac{\partial F_K}{\partial (\mathbf{w}_K)^\prime  } =  \begin{bmatrix}\frac{\partial F_K^x}{\partial (x_K)^\prime} & \frac{\partial F_K^x}{\partial (\bv_K)^\prime } \\[0.2cm]
	\frac{\partial F_K^v}{\partial (x_K)^\prime}  & \frac{\partial F_K^v}{\partial (\bv_K)^\prime }
	\end{bmatrix}_{(n+s) \times (n+s)}$  is a symmetric matrix. This implies, that $F_K$ is a conservative vector field.  
	
	 As  $\alpha_k:[0,1]\rightarrow \mathcal C_k \subset \Omega_k$ is bijective parametrization of a  piece-wise path connecting a fixed point $(x_{0k},\bu_{0k},\bv_{0k})\in \Omega_k$ to an arbitrary point $(x_k,\bu_k,\bv_k)\in \Omega_k$. 
	Following Lemma \ref{lem:potential}, the instantaneous potential function  satisfies
	\begin{align*}
		P_k(x_k,\bu_k,\bv_k)= c_k+\int_{0}^{1}F_k(\alpha_k(z)) \bigcdot \frac{\partial{\alpha_k(z)}}{\partial z}dz,~k \in \mathcal{K}\backslash \{K\}, 
		\end{align*} 
	where $c_{k}=P_k(x_{0k},\bu_{0k},\bv_{0k})$ is value of the potential function evaluated at  $(x_{0k},\bu_{0k},\bv_{0k})$.  Similarly, the terminal potential function is given by
	\begin{align}
	P_K(x_K,\bv_K) = c_{K}+\int_{0}^{1}F_K(\alpha_K(z)) \bigcdot \frac{\partial{\alpha_K(z)}}{\partial
		 z}dz. \notag 
	\end{align}
where $c_{K} = P_K(x_{0K},\bv_{0K})$.
\end{proof}
 
\begin{remark}
From \eqref{eq:Pfk} and \eqref{eq:PfK}, we note that the instantaneous potential function and terminal potential function at a given point are not unique, but are unique up to a constant, and depend upon the choice of the initial fixed points $\{\alpha_k(0),~k\in \mathcal K\}$. This implies that we obtain a family of potential functions, and as a result, several optimal control problems associated with OLPDG. However, as the objective functions of these problems differ by a constant, they have the same optimal solution.
\end{remark} 
%
\section{Open-loop linear quadratic potential difference game}   \label{sec:LQDG}
In this section, we specialize the results obtained from the previous section to a linear quadratic setting and provide a numerical method for computing the open-loop Nash equilibrium associated with OLPDG. Toward this end, we introduce the following $N$-player non-zero sum finite horizon linear quadratic difference game as follows. Each Player $i\in \mathcal N$ solves
\begin{subequations} 
\begin{align}
\text{NZDG1}:~& \min_{\tilde{u}^i,\tilde{v}^i}     J^i(x_0,(\tilde{u}^i,\tilde{u}^{-i}),(\tilde{v}^{i},\tilde{v}^{-i})),\\
&\text{subject to}\notag \\
& x_{k+1}=A_kx_k +\sum_{i\in \mathcal N} B_k^i u_k^i,k\in \mathcal K\backslash \{K\},~ x_0 \text{ (given)}, \label{eq:LQStatedynamics}\\
& M_kx_k+N_k \bv_k +r_k \geq 0,~\bv_k \geq 0,~ k\in \mathcal K, \label{eq:LQconstraints}
\end{align}
where
\begin{align}
J^i(x_0,(\tilde{u}^i,\tilde{u}^{-i}),(\tilde{v}^{i},\tilde{v}^{-i}))& = \frac{1}{2}x_K^{\prime}Q_K^ix_K+p_K^{i^\prime}x_K + \sum_{k=0}^{K-1}\left(\frac{1}{2}x_k^{\prime}Q^i_kx_k+{p_k^i}^\prime x_k+\frac{1}{2}\mathbf{u}^\prime_kR^i_k\mathbf{u}_k\right)\nonumber\\&+\sum_{k=0}^{K}\left(\frac{1}{2}\mathbf{v}_k^{\prime}D^i_k\mathbf{v}_k+{d^{i}_k}^\prime\mathbf{v}_k+x_k^{\prime}L^i_k\mathbf{v}_k \right),
\label{eq:LQobjective}
\end{align}
\end{subequations} 
 where the matrices $Q^i_k\in \mathbb{R}^{n \times n}$,~$i\in \mathcal N$, ~$k\in \mathcal K$ are symmetric, $R_k^i\in \mathbb R^{m_i\times m_i},~i\in \mathcal N,~k\in \mathcal K\backslash \{K\}$ are symmetric and positive definite, $D_k^i\in \mathbb R^{s\times s},~i\in \mathcal N, k\in \mathcal K$ are symmetric,  $p_k^i\in \mathbb R^n$, $i\in \mathcal N$, $k\in \mathcal K \backslash \{K\}$, and $d_k^i\in \mathbb R^s$, $L_k^i\in \mathbb R^{n\times s}$, $i\in \mathcal N,~k\in \mathcal K$.
 Associated with NZDG1 we introduce the following optimal control problem 
 \begin{subequations} 
\begin{align}
    \mathrm{OCP1}:\quad \quad& \min_{\tilde{\mathbf{u}},\tilde{\mathbf{v}}} J(x_0,\tilde{\mathbf{u}},\tilde{\mathbf{v}}),\label{eq:LQOCP1}\\
   &\text{subject to  \eqref{eq:LQStatedynamics} and  \eqref{eq:LQconstraints}}\notag  
   \end{align} 
   where
\begin{align}  J(x_0,\tilde{\mathbf{u}},\tilde{\mathbf{v}})&=\frac{1}{2} x_k ^{\prime}Q_K x_k +p_K^{\prime} x_k  + \sum_{k=0}^{K-1}\left(\frac{1}{2} x_k ^{\prime}Q_k x_k +p^{\prime}_k x_k +\frac{1}{2}\mathbf{u}_k^{\prime}R_k \mathbf{u}_k\right)\nonumber\\& +\sum_{k=0}^{K}\left(\frac{1}{2}\mathbf{v}_k^{\prime}D_k\mathbf{v}_k+d^{\prime}_k\mathbf{v}_k+ x_k ^{\prime}L_k\mathbf{v}_k \right) 
\label{eq:OCP1obj}
\end{align}
\end{subequations} 
with $Q_k\in \mathbb{R}^{n \times n}$, $d_k\in \mathbb R^s$, $p_k\in \mathbb R^n$, $D_k\in \mathbb R^{s\times s}$, $L_k\in \mathbb R^{n\times s}$,
$i\in \mathcal N$, $k\in \mathcal K$, and $R_k\in \mathbb R^{m \times m}$,~$i\in \mathcal N$, $k\in \mathcal K \backslash \{K\}$. 

\begin{assumption}
The admissible action sets $\{U_k^i,~k\in \mathcal K\backslash \{K\},~i\in \mathcal N\}$, are such that the sets of state vectors  $\{X_k,~k\in \mathcal K\}$, obtained from \eqref{eq:LQStatedynamics}, are convex, and feasible action sets $\{V^i_k( x_k ,v_k^{-i})=\{v_k^i\in \mathbb R^{s_i}~|~ M_kx_k+N_k \bv_k+r_k\geq 0,~ \mathbf{v}_k\geq 0\},~\forall x_k\in X_k\}$ are non-empty, convex and bounded for all   $k\in \mathcal K$, $i\in \mathcal N$. 
\label{ass:LQfeasibility}
\end{assumption} 
 
 In the next theorem  we provide conditions under which NZDG1 is an open-loop dynamic potential game, and using Theorem \ref{thm:construction} we construct the potential functions associated  optimal control problem (OCP1).
\begin{theorem}\label{th:LQ Restrictions} Let Assumption  \ref{ass:LQfeasibility} holds true. 	Let the parameters associated with NZDG1 satisfy the following conditions
	\begin{subequations}
			\begin{align}
		&[R_{k}^i]_{ij}=[R_k^j]_{ij}, ~i,j\in \mathcal N,~i\neq j, ~k\in \mathcal K\backslash \{K\}, \label{eq:Ri Rj}\\
		&Q_k^i=Q_k^j,~p_k^i=p_k^j,~L_k^i=L_k^j,[D_k^i]_{ij}=[D_k^j]_{ij}, ~i,j\in \mathcal N,~i\neq j, ~k\in \mathcal K, \label{eq:Di Dj}
		\end{align}
		\label{eq:potcond1}
	\end{subequations}
	then NZDG is a OLDPG. Further, the OCP associated with the related OLDPG is described by
	\begin{subequations}
			\begin{align} \label{eq:PF payoff1}
		&Q_k=Q_k^i,~[D_k]_{i\bullet}=[D_k^i]_{i\bullet},~[d_k]_i=[d_k^i]_i,~L_k=L_k^i,~p_k=p_k^i,
		~i\in \mathcal N,~k\in \mathcal K,\\
		&[R_k]_{i\bullet}=[R_k^i]_{i\bullet},~i\in \mathcal N,~k\in \mathcal K\backslash \{K\}. \label{eq:PF payoff2}
		\end{align}
		\label{eq:ocpcond1}
	\end{subequations}
\end{theorem} 
\begin{proof} We first consider the conditions in \eqref{eq:Di Dj} and verify the condition \eqref{eq: partial x},
	\begin{align*}
	\frac{\partial g_k^i}{\partial x_k} = Q^i_kx_k+p^i_k+L^i_k\mathbf{v}_k =  Q^j_kx_k+p^j_k+L^j_k\mathbf{v}_k = \frac{\partial g_k^j}{\partial x_k},~k\in \mathcal K.
	\end{align*}
	Since 
	$Q_k^i=Q_k^j$, $p_k^i=p_k^j$ and $L_k^i=L_k^j, \forall i,j \in \mathcal{N}$, the condition \eqref{eq: partial x} holds true. Next, by using \eqref{eq:Ri Rj} and \eqref{eq:Di Dj} and the symmetric structure of $R_k^i$ and $D_k^i$,  we verify the conditions \eqref{eq:ux} and \eqref{eq:vxx} as follows 
	\begin{align*}
	&\frac{\partial^2 g_k^i}{\partial (u_k^j)^\prime \partial u_k^i}  =  [R^i_k]_{ij} = [R^j_k]_{ij} =\left([R^j_k]_{ji}\right)^\prime = \left(\frac{\partial^2 g_k^j}{\partial (u_k^i)^\prime \partial u_k^j}\right)^\prime,~k \in \mathcal{K}\backslash\{K\}, \\
	&\frac{\partial^2 g_k^i}{\partial (v_k^j)^\prime \partial v_k^i}  =  [D^i_k]_{ij} = [D^j_k]_{ij}  =\left([D^j_k]_{ji}\right)^\prime  =\left(\frac{\partial^2 g_k^j}{\partial (v_k^i)^\prime \partial v_k^j}\right)^\prime,~k \in \mathcal{K}.
	\end{align*}
	So, NZDG1 is a OLPDG. Next, we proceed to construct the potential functions associated with OLPDG.
	The gradient vector field $F_k=\begin{bmatrix} {F_k^x}^\prime& {F_k^u}^\prime& {F_k^v}^\prime \end{bmatrix}^\prime$ is calculated as
	\begin{align*}
	&F_k^x= Q_k^1x_k+p_k^1+L_k^1\mathbf{v}_k,~	F_k^u= \begin{bmatrix}
[R^1_k]_{1\bullet}\\
[R^2_k]_{2\bullet}\\
	\vdots\\
[R^N_k]_{N\bullet}
	\end{bmatrix} \bu_k,~F_k^v= \begin{bmatrix}
[D^1_k]_{1\bullet}\\
[D^2_k]_{2\bullet}\\
	\vdots\\
[D^N_k]_{N\bullet}
	\end{bmatrix} \bv_k + \begin{bmatrix}[d_k^1]_1\\ [d_k^2]_2\\ \vdots \\ [d_k^N]_N \end{bmatrix} 
	+\begin{bmatrix}
	[L^1_k]^\prime_{\bullet 1}\\
	[L^2_k]_{\bullet 2}\\
	\vdots\\
	[L^N_k]^\prime_{\bullet N}
	\end{bmatrix}x_k.
\end{align*}
Since, $F_k$ is conservative, the instantaneous potential function \eqref{eq:Pfk} evaluated as a line integral, along an arbitrary piecewise path in $\Omega_k$, depends only on the initial and final points. We consider a straight line connecting the origin in $\Omega_k$  and an arbitrary point $(x_k,\bu_k,\bv_k)\in \Omega_k$, and the associated bijective 
parametrization of this line is given by $\alpha_k(z)=z\begin{bmatrix}x_k^\prime & \bu_k^\prime & \bv_k^\prime \end{bmatrix}^\prime$ with $z\in [0, 1]$. The instantaneous potential function is computed as 
	\begin{align}
	P_k(x_k,\bu_k,\bv_k)&= c_k+\int_0^1 F_k (\alpha_k(z)) \bigcdot \frac{d \alpha_k(z)}{dz}~dz \notag \\
	&=c_k+\int_0^1  \left(x_k^\prime ~F_k^x(\alpha_k(z)) + \bu_k^\prime~F_k^u(\alpha_k(z))+\bv_k^\prime~ F_k^v(\alpha_k(z))\right) dz.  
	\label{eq:Pkeq}
	\end{align}
	We define matrices $R_k$, $D_k$, $Q_k$, $L_k$, $p_k$ and $d_k$  such that $[R_k]_{i\bullet}=[R_k^i]_{i\bullet}$,  
	$[D_k]_{i\bullet}=[D_k^i]_{i\bullet}$, $Q_k=Q_k^i$, $[L_k]_{\bullet i}=[L_k^i]_{\bullet i}$, $p_k=p_k^i$, and $[d_k]_i=[d_k^i]_i$ for all $i\in \mathcal N$, $k\in \mathcal K\backslash \{K\}$. Then, from \eqref{eq:potcond1} it follows that $R_k$ and $D_k$ are symmetric matrices and $Q_k=Q_k^i=Q_k^j$, $L_k=L_k^i=L_k^j$, $p_k=p_k^i=p_k^j$ for all $i,j\in \mathcal N$, $k\in \mathcal K \backslash \{K\}$. Using this \eqref{eq:Pkeq} can be written as
	\begin{multline}
P_k(x_k,\bu_k,\bv_k)=c_k+ \int_0^1 \left(x_k^\prime \left[Q_k (zx_k)+p_k+L_k (z\bv_k)\right]+\bu_k^\prime R_k (z\bu_k)+
\bv_k^\prime \left[ D_k (z\bv_k) + d_k + L_k (zx_k)\right] \right) dz  \\ 
=c_k+ \frac{1}{2}x_k^\prime Q_k x_k+p_k^\prime x_k +\frac{1}{2} \bu_k^\prime R_k \bu_k +\frac{1}{2} \bv_k^\prime D_k \bv_k + d_k^\prime \bv_k + x_k^\prime L_k \bv_k.
\label{eq:potfn1}
	\end{multline}
Similarly, the terminal vector field $F_K=\begin{bmatrix}{F_K^x}^\prime & {F_K^v}^\prime \end{bmatrix}^\prime$ is calculated as 
\begin{align*}
F_K^x= Q_K^1x_K+p_K^1+L_K^1\mathbf{v}_K,~
F_K^v= \begin{bmatrix}
[D^1_K]_{1\bullet}\\
[D^2_K]_{2\bullet}\\
\vdots\\
[D^N_K]_{N\bullet}
\end{bmatrix} \bv_K + \begin{bmatrix}[d_K^1]_1\\ [d_K^2]_2\\ \vdots \\ [d_K^N]_N \end{bmatrix} 
+\begin{bmatrix}
[L^1_K]^\prime_{\bullet 1}\\
[L^2_K]_{\bullet 2}\\
\vdots\\
[L^N_K]^\prime_{\bullet N}
\end{bmatrix}x_K.
\end{align*}
	We define matrices  $D_K$, $Q_K$, $L_K$, $p_K$ and $d_K$  such that 
$[D_K]_{i\bullet}=[D_K^i]_{i\bullet}$, $Q_K=Q_K^i$, $[L_K]_{\bullet i}=[L_K^i]_{\bullet i}$, $p_K=p_K^i$, and $[d_K]_i=[d_K^i]_i$ for all $i\in \mathcal N$. Then, from \eqref{eq:potcond1} it follows that  $D_K$ is a symmetric matrix and $Q_K=Q_K^i=Q_K^j$, $L_K=L_K^i=L_K^j$, $p_K=p_K^i=p_K^j$ for all $i,j\in \mathcal N$. Using this, and using the same procedure as before, the terminal potential function \eqref{eq:PfK} is calculated as 
\begin{align}
P_K(x_K,\bv_K)&=c_K+\int_0^1 \left(x_K^\prime F_K^x(\alpha_K(z))\bv_K^\prime F_K^v(\alpha_K(z)) \right) dz \notag\\
&=c_K+\frac{1}{2}x_K ^\prime Q_K x_K+ p_K^\prime x_K  +\frac{1}{2}\bv_K^\prime D_K \bv_K + x_K^\prime L_K \bv_K+ d_K^\prime \bv_K.
\label{eq:potfn2}
\end{align}
The instantaneous and terminal potential function given by \eqref{eq:potfn1} and \eqref{eq:potfn2}, respectively, constitute the objective function  \eqref{eq:OCP1obj} associated with the OCP1. The parameters associcated with this objective function satisfy the conditions \eqref{eq:ocpcond1}.
\end{proof}
\subsection{Computation of open-loop Nash equilibrium associated with OLPDG}
\label{sec:solvability}
In this section, under a few assumptions on the parameters we transform the necessary conditions associated with OCP1 to a large-scale linear complementarity problem, there by providing a way to compute the open-loop Nash equilibrium. 
Let $(\mathbf{\tilde{u}^{*},\tilde{v}^{*}})$ be the optimal solution of the OCP1 and $\{x_k ^*,~k\in \mathcal K\}$ be the state trajectory generated by $\mathbf{\tilde{u}^{*}}$. The necessary conditions of optimality of
 are then given by
\begin{subequations}
\begin{align}
\text{for }& k\in \mathcal K\backslash \{K\} \notag \\
\label{eq:LQOCP Cond1}  &  R_k\mathbf{u}_k^*+ \mathbf{B}_k^\prime\lambda^*_{k+1} = 0,~ \mathbf{B}_k=\begin{bmatrix}B_k^1 &   B_k^1 & \cdots &  B_k^N \end{bmatrix}, \\
\label{eq:LQOCP Cond2}  & \mathbf{x}_{k+1}^{*} = A_k x_k ^*+\sum_{l \in \mathcal N}B_k^lu_k^{l^*}, \quad x_{0}~ \text{is given},\\
\label{eq:LQOCP Cond3}  & \lambda_k^* = Q_k x_k ^{*}+p_k+L_k\mathbf{v}_k^{*}+ A_k^\prime \lambda_{k+1}^{*}-M_k^{\prime}\mu_k^{*},\\
\label{eq:LQOCP Cond4}  & \lambda_K^{*} = Q_K x_k ^*+p_K+L_K\mathbf{v}_K^*-M_K^{\prime}\mu_K^*,\\
\text{for }& k\in \mathcal K \notag \\
\label{eq:LQOCP Cond5} & 0 \leq \left( D_k\mathbf{v}_k^{*}+d_k+L_k^\prime x_k ^* - N_k^\prime\mu_k^*\right) \perp v_k^* \geq 0, \\
\label{eq:LQOCP Cond7}  & 0 \leq \left( M_k x_k ^*+ N_k\mathbf{v}_k^*+r\right) \perp \mu_k^* \geq 0.
 \end{align}
 \end{subequations} 
The above set of necessary conditions lead to a weakly coupled system of a parametric two-point boundary value problem \eqref{eq:LQOCP Cond1}-\eqref{eq:LQOCP Cond4} and a parametric linear complimentarity problem 
\eqref{eq:LQOCP Cond5}-\eqref{eq:LQOCP Cond7}. We have the following assumption.
\begin{assumption}\label{asm:Affine Co-state}
The co-state variable $\lambda_k$ is assumed to be affine in the state variable $ x_k $ for $k \in \mathcal K$ i.e., $\lambda_k^* = H_k x_k ^*+\beta_k$  where $H_k \in \mathbb{R}^{n \times n}$ and $\beta_k \in \mathbb{R}^{n \times 1}$.
\end{assumption} 
Using Assumption \ref{asm:Affine Co-state} it can be shown that the two-point boundary value problem \eqref{eq:LQOCP Cond1}-\eqref{eq:LQOCP Cond4} can be solved if the following backward equations for $k \in \mathcal K\backslash \{K\}$ and $i \in \mathcal N$ has a solution
\begin{subequations} 
\begin{align}
    &\label{eq: Backward1} \Gamma_{k+1}^{k} = \mathbf{I}+S_k H_{k+1}, \\
    & \label{eq: Backward2} H_{k}=Q_k+A_k^\prime H_{k+1}(\Gamma^{k}_{k+1})^{\mbox{-}1}A_k,\\
    &\label{eq: Backward3} \beta_k =p_k-M^{\prime}\mu_k^*+L_k\mathbf{v}_k^*+A_k^{\prime}\beta_{k+1}-A_k^{\prime}H_{k+1}(\Gamma_{k
    +1}^k)^{\mbox{-}1}S_k\beta_{k+1},
\end{align} 
\label{eq:RDE}
\end{subequations}
where $S_k = \mathbf{B}_kR_k^{\mbox{-}1}\mathbf{B}_k^{\prime}$, $H_K = Q_K$ and $\beta_K = p_K+L_K\mathbf{v}_K^*-M_K^{\prime}\mu_K^*$. Assuming $\Gamma_{k+1}^{k}$ to be invertible for $k = K-1,\dots,1$, we obtain $H_k$ and $\beta_k$ for $k = K-1,\dots,0$, and the state vector $ x_k ^*$ and the joint control vector  $\mathbf{u}_k^*$ are given by
\begin{subequations}
\begin{align}
\label{eq:State DE} \mathbf{x}_{k+1}^* & = \left(\Gamma^k_{k+1} \right)^{\mbox{-}1}\left(A_k x_k ^*-S_k\beta_{k+1} \right), ~k \in \mathcal K\backslash \{K\}, \\
\label{eq:decision Eq}   \mathbf{u}_k^* & = -R_k^{\mbox{-}1}\mathbf{B}_k^{\prime}\left(H_{k+1}\mathbf{x}_{k+1}^{*}+\beta_{k+1} \right).
\end{align} 
\label{eq:statecontrol} 
\end{subequations} 
\begin{remark} 
Suppose that the set of backward equations \eqref{eq: Backward1}-\eqref{eq: Backward3} admits a solution, i.e., the matrix $\Gamma^k_{k+1}$ is invertible for all $k \in \mathcal K \backslash \{K\}$, then the two point boundary value problem in \eqref{eq:LQOCP Cond2}-\eqref{eq:LQOCP Cond4} has a unique solution. To show this, let $\bar{\lambda_k} = \lambda_k^* - (H_k x_k ^*+\beta_k)$ be another solution for the two point boundary value problem \eqref{eq:LQOCP Cond2}-\eqref{eq:LQOCP Cond4}. Substituting $\lambda_k^* = \bar{\lambda}_k+ H_k x_k ^*+\beta_k$ in \eqref{eq:LQOCP Cond2} and \eqref{eq:LQOCP Cond3}, we get a decoupled system of equations as  
\begin{subequations}
\begin{align}
\mathbf{x}_{k+1}& = (\Gamma^k_{k+1})^{\mbox{-}1}\left(A_k\mathbf{x}_{k}^*-S_k\bar{\lambda}_{k+1}-S_k\beta_{k+1} \right),\\
\bar{\lambda}_{k} & = A_k^{\prime}\bar{\lambda}_{k+1}-A_k^{\prime}H_{k+1}(\Gamma^k_{k+1})^{\mbox{-}1}S_k\bar{\lambda}_{k+1}.
\end{align}	
\end{subequations}
  From the terminal condition, $\bar{\lambda}_{K}=0$ which results in $\bar{\lambda}_{k}=0 ~\forall k \in \mathcal K$. This proves that the solution for the two point boundary value problem described by \eqref{eq:LQOCP Cond2}-\eqref{eq:LQOCP Cond4} is unique. 
\label{rem:unique}
\end{remark} 

 In view of Remark \ref{rem:unique} we have the following standing assumption in the remaining part of the paper.
\begin{assumption}\label{asm:Gamma Inv}
The set of matrices $\{\Gamma^k_{k+1},~ k \in \mathcal K\backslash \{K\}\}$ are invertible.
\end{assumption}\noindent
Next, we note that the backward equations \eqref{eq: Backward1} and \eqref{eq: Backward2} are coupled but evolve independently of \eqref{eq: Backward3}. Taking $G_{k+1} = A_k^{\prime}-A_k^{\prime}H_{k+1}(\Gamma_{k
	+1})^{\mbox{-}1}S_k$,  \eqref{eq: Backward3} can be represented in the vector form as follows 
    \begin{align} \label{eq:Beta DE}
        \beta_k = \sum_{\tau = k}^{K}\psi(k,\tau)\left( p_\tau+\begin{bmatrix}L_\tau &-M_\tau^{\prime}\end{bmatrix} \begin{bmatrix}\mathbf{v}_\tau^{*^\prime} & \mu_\tau^{*^\prime}\end{bmatrix}^{\prime}\right),~\forall k \in \mathcal K,
    \end{align} where the transition matrix associated with \eqref{eq: Backward3} is  $\psi(k,\tau)=G_{k+1}\cdots G_{\tau-1}G_{\tau}$, if $\tau > k$ and $\psi(k,k) = \mathbf{I}$. Thus,  \eqref{eq:Beta DE} represents a parametric linear backward difference equation parametrized by $\{\mathbf{v}_k^*,\mu_k^*,~k \in \mathcal K, ~i \in \mathcal N\}$. Now, we analyse the forward difference equation for the state trajectory \eqref{eq:State DE}. Suppose
    \begin{align*}
        \mathbf{x}_{k+1} = \bar{A}_{k} x_k ^*+\bar{B}_{k}\beta_{k+1}, \quad x_0~ \text{is given},
    \end{align*} where $\bar{A}_k = (\Gamma^k_{k+1})^{\mbox{-}1}A_k$ and $\bar{B}_k = -(\Gamma^k_{k+1})^{\mbox{-}1}S_k$ $\forall k \in \mathcal K\backslash\{K\}$. Denoting the transition matrix as $\phi(\rho,k) = \bar{A}_{k-1}\bar{A}_{k-2}\cdots \bar{A}_{\rho}$ for $\rho < k$ and $\phi(k,k) = \mathbf{I}$, and from \eqref{eq:Beta DE}  we have  
   for $k \in \mathcal K\backslash\{0\}$, 
\begin{multline}
x_k ^* = \phi(0,k)x_0+\sum_{\tau= 1}^{K}\left(\left(\sum_{\rho=1}^{\operatorname{min}(k,\tau)}\phi(\rho,k)\bar{B}_{\rho-1}\psi(\rho,\tau) \right)\left( p_\tau+\begin{bmatrix}L_\tau &-M_\tau^{\prime}\end{bmatrix} \begin{bmatrix}\mathbf{v}_\tau^{*^\prime} & \mu_\tau^{*^\prime}\end{bmatrix}^{\prime}\right)\right).
\label{eq:State ForwardEq2}
\end{multline} 
Further, we combine the variables in \eqref{eq:State ForwardEq2} as $p_{\mathbf{K}} = \begin{bmatrix}p_{1}^{\prime}& \hdots & p_{K}^{\prime} \end{bmatrix}^{\prime} $, $~\mathbf{x}_{\mathbf{K}}^* = \begin{bmatrix}\mathbf{x}_{1}^{*^\prime}& \hdots ~ \mathbf{x}_{K}^{*^\prime} \end{bmatrix}^{\prime} $ and  $y_{\mathbf{K}}^{*} =\left[\mathbf{v}_{1}^{*^\prime}~\mu_1^{*^\prime}\right.$ \\ $\left.\hdots ~ \mathbf{v}_{K}^{*^\prime}~\mu_K^{*^\prime} \right]^{\prime}$. As a result,   \eqref{eq:State ForwardEq2} can be written as:
\begin{align}\label{eq:State ForwardEq3}
     {x}_{\mathbf{K}}^* =\mathbf{\Phi}_{0}x_0+\mathbf{\Phi}_1p_{\mathbf{K}}+\mathbf{\Phi}_2y_{\mathbf{K}}^*,
\end{align} \noindent
where $[\mathbf{\Phi}_0]_k = \phi(0,k),~ [\mathbf{\Phi}_1]_{k\tau} =\sum_{\rho=1}^{\operatorname{min}(k,\tau)}\phi(\rho,k)\bar{B}_{\rho-1}\psi(\rho,\tau) $ and  $[\mathbf{\Phi}_2]_{k\tau} = \sum_{\rho=1}^{\operatorname{min}(k,\tau)}\phi(\rho,k)\bar{B}_{\rho-1}\psi(\rho,\tau)\\ \begin{bmatrix}L_\tau &-M_\tau^{\prime}\end{bmatrix}$ for $k,\tau \in \mathcal K\backslash \{0\}.$ Thus, we expressed the state trajectory paramterized with $\{\mathbf{v}_k^*,\mu_k^*,~k \in \mathcal K\backslash \{0\}\}$. Next, we analyse the parametric linear complimentarity problem in \eqref{eq:LQOCP Cond5}-\eqref{eq:LQOCP Cond7} in detail. First, the vector representation of these problems is given by
\begin{align}\label{eq:pLCP1}
 \mathrm{pLCP}( x_k ^*):~\begin{bmatrix}D_k & -N_k^{\prime} \\ N_k & 0 \end{bmatrix}\begin{bmatrix}\mathbf{v}_k^* \\ \mu_k^* \end{bmatrix}+\begin{bmatrix}L_k^{\prime} \\ M_k\end{bmatrix} x_k ^* + \begin{bmatrix}d_k \\ r_k \end{bmatrix}  \perp \begin{bmatrix}\mathbf{v}_k^* \\ \mu_k^* \end{bmatrix}.
\end{align} 
Let $\tilde{\mathbf{M}} = \begin{bmatrix}D_1 & -N_1^{\prime} \\ N_1 & 0 \end{bmatrix}\oplus \cdots \oplus \begin{bmatrix}D_K & -N_K^{\prime} \\ N_K & 0 \end{bmatrix} , \tilde{\mathbf{q}} =\begin{bmatrix}L_1^{\prime} \\ M_1\end{bmatrix} \oplus \cdots \oplus  \begin{bmatrix}L_K^{\prime} \\ M_K\end{bmatrix}$ and $\tilde{\mathbf{s}} =  \begin{bmatrix}d_1^\prime &r_1^\prime &\cdots &d_K^\prime& r_K^\prime \end{bmatrix}^\prime$. 
Aggregating these parametric problems for all time steps $k \in \mathcal K \backslash \{0\}$, we obtain a single parametric linear complementarity problem as 
\begin{align}\label{eq:pLCP2}
   \mathrm{pLCP(\mathbf{x}_\mathbf{K}^*}):\quad  \tilde{\mathbf{M}}y_{\mathbf{K}}^*+\tilde{\mathbf{q}}\mathbf{x}_{\mathbf{K}}^* + \tilde{\mathbf{s}} \perp y_{\mathbf{K}}^*.
\end{align}
Substituting \eqref{eq:State ForwardEq3} in \eqref{eq:pLCP2} results in the following large scale linear complimentarity problem:
\begin{align}\label{eq:LCP1}
   \mathrm{LCP}(x_0):\quad  \mathbf{M}y_{\mathbf{K}}^*+\mathbf{q} \perp y_{\mathbf{K}}^*,
\end{align}\noindent
where $\mathbf{M} = \tilde{\mathbf{M}} + \tilde{\mathbf{q}}\mathbf{\Phi}_2$ and $\mathbf{q} =\tilde{\mathbf{q}}(\mathbf{\Phi}_{0}x_0+\mathbf{\Phi}_1p_{\mathbf{K}})+\tilde{\mathbf{s}}$.
Thus, we formulated the necessary conditions \eqref{eq:OCP Cond1}-\eqref{eq:OCP Cond7} as a single large scale linear complimentarity problem with the aid of Assumptions \ref{asm:Affine Co-state} and \ref{asm:Gamma Inv}. Solving \eqref{eq:LCP1} we obtain candidate optimal solution of OCP1 there by a candidate open-loop Nash equilibrium of NZDG1.

Next, we study sufficient conditions under which the solution of $\mathrm{LCP}(x_0)$ and $\mathrm{pLCP}(x_0)$ indeed minimize OCP1, and as a result provide an open-loop Nash equilibrium of NZDG1. The sufficient conditions provided in Theorem \ref{thm:suff1} requires both the minimized instantaneous and terminal Lagrangian functions to be convex in the state variables. Since the solutions $\bv_k^*$ are obtained from solving the parametric linear complementarity problem \eqref{eq:pLCP2}, upon substitution, the minimized Lagrangian functions may not be convex in the state variables. So, to derive the required sufficient conditions we transform the OCP1 as static optimization problem in the decision variables $(\tilde{\bu},\tilde{\bv})$. 
Toward this end, the objective function associated with the OCP1 can be written as
\begin{align}
 J(x_0,\tilde{\mathbf{u}},\tilde{\mathbf{v}})&=\sum_{k=0}^{K-1}\frac{1}{2}\bu_k^\prime R_k \bu_k +\sum_{k=0}^K\frac{1}{2}\left( x_k ^{\prime}Q_k x_k +\left(p_k+L_k\bv_k^{*}
-M_k^\prime \mu_k^*\right)^\prime x_k \right) \notag\\ 
&+\sum_{k=0}^{K}\left(\frac{1}{2}\mathbf{v}_k^{\prime}D_k\mathbf{v}_k+d^{\prime}_k\mathbf{v}_k+ x_k ^{\prime}L_k (\bv_k-\bv_k^*)+{\mu_k^*}^\prime M_k x_k \right),
\label{eq:objOCP}    
\end{align} 
where $\{(\mathbf{v}_k^*,\mu_k^*),~k \in \mathcal K\}$ is the solution of  $\mathrm{LCP}(x_0)$ and $\mathrm{pLCP}(x_0)$.
 We define value function $W_k,~k\in \mathcal K$ as
\begin{align*}
& W_k=\frac{1}{2}x_k^\prime E_k x_k + e_k^\prime x_k + w_k,
\end{align*}
where $E_k \in \mathbb{R}^{n \times n},~e_k \in \mathbb{R}^n,~w_k \in \mathbb{R}$ and $i \in \mathcal{N}$. Let the matrices $T_k:=R_k+ \mathbf{B}_k^\prime  E_{k+1}\mathbf{B}_k$ be invertible for all $k\in \mathcal K\backslash \{K\}$  with the matrices $E_k,~ k\in \mathcal K$ computed as the solution of the following backward Ricatti difference equations
\begin{subequations} 
	\begin{align}
	&E_k= A_k^\prime E_{k+1} A_k +Q_k -A_k^\prime E_{k+1}\mathbf{B}_k T_k^{-1} \mathbf{B}_k^\prime E_{k+1}A_k,~ E_K=Q_K,\label{eq:LQDGSC}\\
	&e_k=A_k^\prime e_{k+1}-A_k^\prime E_{k+1}\mathbf{B}_k T_k ^{-1}\mathbf{B}_k^\prime e_{k+1}+p_k+L_k\bv_k^* -M_k^\prime \mu_k^*,~e_K=p_K+L_K\bv_K^*-M_K^\prime \mu_K^*, \label{eq:lqdge}\\
	&w_k=w_{k+1}-e^\prime_{k+1} \mathbf{B}_kT_k^{-1}\mathbf{B}_k^\prime e_{k+1},~w_K=0.
	\end{align}  
	\label{eq:ric2}
\end{subequations} 
Then, by denoting  $\Delta_k = W_{k+1}-W_{k}$, and  using the sum $\sum_{K=0}^{K-1}\Delta_k$ and  \eqref{eq:ric2}  we can write the objective function \eqref{eq:objOCP} as  
\begin{align} 
 J(x_0,\tilde{\mathbf{u}},\tilde{\mathbf{v}})& =   W_0 +\sum_{k=0}^{K-1} \frac{1}{2}||\mathbf{u}_k  +T_k^{-1} \mathbf{B}_k^\prime\left( E_{k+1}A_kx_k+ e_{k+1}\right)  ||^2_{T_k}\notag  \\ 
 &+\sum_{k=0}^{K}\left(\frac{1}{2}\mathbf{v}_k^{\prime}D_k\mathbf{v}_k+d^{\prime}_k\mathbf{v}_k+ x_k ^{\prime}L_k (\bv_k-\bv_k^*)+{\mu_k^*}^\prime M_k x_k \right).\label{eq:static}
\end{align} 

The next lemma relates how the candidate  optimal control \eqref{eq:decision Eq}  obtained by solving  the two-point boundary value problem \eqref{eq:LQOCP Cond2}-\eqref{eq:LQOCP Cond4} is related to the minimizer of  the objective function  \eqref{eq:static}.

\begin{lemma} Let the set of matrices $\{T_k,~k\in \mathcal K\backslash \{K\}\}$ be invertible and 
	the solutions $E_k$ of the symmetric matrix Riccati difference equation  
	\eqref{eq:LQDGSC} exist for all $k\in \mathcal K$. If the two-point boundary value problem 
	\begin{subequations} 
	\begin{align}
	\bar{\lambda}_k &= A^\prime_k \bar{\lambda}_{k+1}+Q_k \bar{x}_k+p_k+L_k\mathbf{v}_k^{*} -M^{\prime}_k\mu_k^{*},\label{eq:lbareq1}\\ \bar{\lambda}_K &= Q_K x_k+p_K+L\mathbf{v}_K^*-M_K^{\prime}\mu_K^*,\label{eq:lbareq2}\\
	\bar{x}_{k+1}&=A_k \bar{x}_k - \mathbf{B}_kR_k^{-1} \mathbf{B}_k^\prime \bar{\lambda}_{k+1},~\bar{x}_0=x_0	
	\end{align}
	\label{eq:lemtbvp}
	\end{subequations} 
	has a unique solution, then we set $\bar{\bu}_k=-{R}^{-1}_k \mathbf{B}_k^\prime \bar{\lambda}_{k+1}$ and $\bar{e}_k:=\bar{\lambda}_k-E_k \bar{x}_k$. Then the sequences
	$\{\bar{x}_k,  \bar{e}_k\}$ solve equations $\bar{\mathbf{u}}_k+T_k^{-1}\mathbf{B}_k^\prime \left(E_{k+1}A_k \bar{x}_k+\bar{e}_{k+1}\right)=0$ and \eqref{eq:lqdge}.
	\label{lem:NCtoSC}
\end{lemma}
\begin{proof}
	Firstly, we have
	\begin{align*}
	\bar{\bu}_k+T_k^{-1} \mathbf{B}_k^\prime\left( E_{k+1} A_k \bar{x}_k+\bar{e}_{k+1}\right) &=-R_k^{-1} \mathbf{B}_k^\prime\bar{\lambda}_{k+1}+
	T_k^{-1}\mathbf{B}_k^\prime E_{k+1}
	\left(\bar{x}_{k+1}+\mathbf{B}_k R_k^{-1}\mathbf{B}_k^\prime \bar{\lambda}_{k+1}\right)+T_k^{-1}\mathbf{B}_k^\prime \bar{e}_{k+1}\\
	 &= \left(-R_k^{-1}+T_k^{-1}\left(
	R_k+\mathbf{B}_k ^\prime E_{k+1}\mathbf{B}_k\right)R_k^{-1} \right) \mathbf{B}_k^\prime \bar{\lambda}_{k+1}=0.
	\end{align*}
	To prove \eqref{eq:lqdge} we have
	\begin{align*}
&	A^\prime_k \bar{e}_{k+1}-A^\prime_k E_{k+1}\mathbf{B}_k T_k^{-1}\mathbf{B}_k^\prime\bar{e}_{k+1}+p_k+L_k\bv_k^*-M_k^\prime \mu_k^*-\bar{e}_k \\ & \hspace{1.5in}	=p_k+L_k\bv_k^*-M_k^\prime \mu_k^*+A_k^\prime \bar{\lambda}_{k+1}-\bar{e}_{k}-A_k^\prime E_{k+1} \mathbf{B}_k T_{k}^{-1}\mathbf{B}_k^\prime \bar{\lambda}_{k+1}\\& \hspace{1.5in}+
	A_k^\prime E_{k+1}\left(\mathbf{B}_k T_k^{-1}\mathbf{B}_k^\prime E_{k+1}-\mathbf{I}\right) \left(A_k\bar{x}_k-\mathbf{B}_k R_k^{-1}\mathbf{B}_k^\prime \bar{\lambda}_{k+1}\right)\\ &\hspace{1.5in}
=\left(p_k+L_k\bv_k^*-M_k^\prime \mu_k^*\right)+A_k^\prime \bar{\lambda}_{k+1}+Q_kx_k-\bar{\lambda}_k \\
	&\hspace{1.5in}-\left(A_k^\prime E_{k+1}A_k-A_k^\prime E_{k+1}
	\mathbf{B}_k T_k^{-1} \mathbf{B}_k^\prime E_{k+1}A_k+Q_k-E_k \right)x_k\\
&\hspace{1.5in}+A_k^\prime E_{k+1}B_k \left(R_k^{-1}-T_k^{-1}
	-T_k^{-1}\mathbf{B}_k^\prime E_{k+1} \mathbf{B}_k R_k^{-1}\right) B_k^\prime \bar{\lambda}_{k+1}=0.
	\end{align*}
	The last step in the above expression is obtained by using \eqref{eq:LQDGSC} and \eqref{eq:lbareq1},\eqref{eq:lbareq2}. 
 \end{proof}

In the following lemma we provide conditions under which the objective function \eqref{eq:static} is a strictly convex function of $(\tilde{\bu},\tilde{\bv})$.
\begin{lemma} 
Let the solution $E_k$ of the symmetric Ricatti equation 	\eqref{eq:LQDGSC} exist. Let 
\begin{align}
\Upsilon_{k}^{\tau}  =
\begin{cases} 
\mathbf{B}_\tau E_{\tau+1} A_{\tau}A_{\tau-1}\dots A_{k+1} &0 \leq k < K-1 \\ 
\mathbf{0}& k = K-1  
\end{cases}
\end{align} Now, we  define the matrix
\begin{align}
\mathbf{H}&=\begin{bmatrix} \mathbf{Y} & \mathbf{C}^\prime  \\ \mathbf{C} & \mathbf{D} \end{bmatrix},
\label{eq:Hessian}
\end{align}
\begin{multline*} 
where 
~\mathbf{D}=\oplus_{k=0}^{K} D_k,\quad 
{\mathbf{C} = \begin{bmatrix}
	\mathbf{0} &  \mathbf{B}_0^\prime L_1 & \mathbf{B}_0^\prime A_1^\prime  L_2 &  \hdots & \mathbf{B}_0^\prime A_{1}^\prime  \cdots  A_{K-2}^\prime  A_{K-1}^\prime L_K\\ \mathbf{0} & \mathbf{0} & \mathbf{ B}_1^\prime L_2 & \hdots &  \mathbf{B}_1^\prime A_{2}^\prime \cdots A_{K-2}^\prime A_{K-1}^\prime L_K \\ \mathbf{0} & \mathbf{0} & \mathbf{0} & \hdots & \mathbf{B}_2^\prime A_{3}^\prime \cdots A_{K-2}^\prime A_{K-1}^\prime L_K \\ 
	\vdots & \vdots & \vdots & \mathbf{0} & \ddots & \vdots \\
	\mathbf{0} & \mathbf{0} & \mathbf{0} &    \hdots & \mathbf{B}_{K-1} ^\prime L_K
	\end{bmatrix}}
\end{multline*}
and $\mathbf{Y}$ is a $Km \times Km$ matrix  where each block submatrix $[Y]_{lk}$ is a $m \times m$  matrix such that for $k \in \mathcal{K}\backslash \{K\}$
\begin{align*}
[Y]_{lk}& = \begin{cases} \mathbf{B}_k^\prime E_{k+1}A_k\dots A_{l+1}\mathbf{B}_l+ \mathbf{B}_k^\prime \left( \sum_{\tau = k+1}^{K-1}( \Upsilon_{k}^{\tau})^\prime T_{
	\tau}^{-1}( \Upsilon_{k}^{\tau}) A_k\dots A_{l+1}\right)\mathbf{B}_{l},~0 \leq l < k, \\ 
T_k + \mathbf{B}_k^\prime \left[\sum_{\tau = k+1}^{K-1}(  \Upsilon_{k}^{\tau})^\prime T_{
	\tau}^{-1}(  \Upsilon_{k}^{\tau}) \right]\mathbf{B}_k,~ l = k, 
\end{cases} \\ [Y]_{kl}& = [Y]_{lk}^\prime.
\end{align*} 
If the matrix $\mathbf{H}$ is positive definite then the objective function \eqref{eq:static} is
a strictly convex function of $(\tilde{\bu},\tilde{\bv})$.
\label{eq:convexity}
\end{lemma} %
\begin{proof}
We compute the Hessian matrix of the objective function \eqref{eq:static} with respect the decision variables $(\tilde{\bu},\tilde{\bv})$ by eliminating the state variable $x_k$. Now, we have 
$\frac{\partial^2 J}{\partial (\bv_k)^\prime \partial \bv_k}=D_k$,~$\frac{\partial^2 J}{\partial (\bv_k)^\prime \partial v_l}=0$, for $k\neq l$, $k,l\in \mathcal K$. Calculating the remaining second order partial derivatives we get  
\begin{align*}
\quad [Y]_{lk} &=  \frac{\partial^2 J}{\partial (\bu_l)^\prime \partial \bu_k}   = \mathbf{B}_k^\prime E_{k+1}A_k\dots A_{l+1}\mathbf{B}_l  + \mathbf{B}_k^\prime \left[ \sum_{\tau = k+1}^{K-1}( \Upsilon_{k}^{\tau})^\prime T_{
	\tau}^{-1}(  \Upsilon_{k}^{\tau}) \right] A_k\dots A_{l+1}\mathbf{B}_{l},\; 0 \leq l < k, \\
[Y]_{kk}& =   \frac{\partial^2 J}{\partial (\bu_k)^\prime \partial \bu_k} = T_k + \mathbf{B}_k^\prime \left[\sum_{\tau = k+1}^{K-1}(  \Upsilon_{k}^{\tau})^\prime T_{
	\tau}^{-1}(  \Upsilon_{k}^{\tau}) \right]\mathbf{B}_k. \\
	[\mathbf{C}]_{lk}& = \frac{\partial^2 J}{\partial (\mathbf{u}_l)^\prime \partial \mathbf{v}_k} = \begin{cases} \mathbf{B}_l^\prime A_{l+1}^\prime \cdots A_{k-1}^\prime L_{k},~ l < k-1 \\
\mathbf{B}_{l}^\prime L_{k},~ l = k-1~\text{and}\\ 
\mathbf{0},~ l \geq k
\end{cases}
\end{align*}
Then, if the Hessian matrix $\mathbf{H}$ is positive definite, the objective function \eqref{eq:static} is strictly convex in $(\tilde{\bu},\tilde{\bv})$.
 \end{proof}

Next, using the above result in the next theorem we show that the solutions of $\mathrm{LCP}(x_0)$ and
$\mathrm{pLCP}(x_0)$ indeed provide the optimal solution of the OCP1. 
\begin{theorem}\label{thm:LCPOLNE}
	Let Assumptions  \ref{ass:LQfeasibility}, \ref{asm:Affine Co-state} and
	 \ref{asm:Gamma Inv} hold true. Let the Hessian matrix $\mathbf{H}$ given by
	\eqref{eq:Hessian} is positive definite. Then the solutions of $\mathrm{LCP}(x_0)$ and $\mathrm{pLCP}{(x_0)}$ constitute an open-loop Nash equilibrium of the NZDG1.	
\end{theorem}
\begin{proof}
	Let $\{(\mathbf{v}_k^*,\mu_k^*),~k \in \mathcal K\}$ be the solution of $\mathrm{LCP}(x_0)$ and $\mathrm{pLCP}(x_0)$.
	Then transforming the objective function of the OCP1 as \eqref{eq:static} and consider the minimization
 problem subject to the state dynamics $x_{k+1}=A_kx_k +\mathbf{B}_k \bu_k$ and the constraints $M_k x_k+N_k \bv_k+r_k \geq 0$, $\bv_k\geq 0$.  Since $\mathbf{H}$  is positive definite, we have that the objective function  $J(x_0,\tilde{\bu},\tilde{\bv})$ is strictly convex in $(\tilde{\bu},\tilde{\bv})$ from Lemma \ref{eq:convexity}. Also, from Assumption \ref{asm:feasibility}, the sets $\{U_k^i,~k\in \mathcal K \backslash \{K\}\}$ and $\{V_k^i,~k\in \mathcal K\}$ for all $i\in \mathcal N$ are non-empty, convex and bounded. Therefore, by solving the KKT conditions, we obtain the solution of the static optimization. The Lagrangian associated with this optimization problem is given by
    \begin{align}
    \mathcal L= J(x_0,\tilde{\bu},\tilde{\bv})-\sum_{k=0}^{K}\mu_k^\prime \left(M_k x_k +N_k \bv_k+r_k\right).
    \end{align}
	The KKT conditions are then given by
	\begin{subequations}
	\begin{align}
&T_k \left(\bu_k+T_k^{-1}\mathbf{B}_k^\prime (E_{k+1}A_kx_k+e_{k+1})\right)\notag+\mathbf{B}_k^\prime \sum_{\tau = k+1}^{K-1} \left(( \Upsilon_{k}^{\tau})^{\prime }\left(\bu_\tau + T_\tau^{-1} \mathbf{B}_\tau^\prime ( E_{\tau+1}A_\tau x_\tau +e_{\tau+1}\right)\right)\\& +\mathbf{B}_k^\prime \left(L_{k+1} (\bv_{k+1}-\bv_{k+1}^*)-M_{k+1}^\prime (\mu_{k+1}-\mu_{k+1}^*)
	\right)+\sum_{\tau=k+1}^{K-1}\left(A_{\tau}A_{\tau-1}\cdots A_{k+1} \mathbf{B}_k\right)^\prime \left(
	L_{\tau+1}(\bv_{\tau+1}-\bv_{\tau+1}^*)\right)\notag\\&-\sum_{\tau=k+1}^{K-1}\left(M^\prime_{\tau+1}(\mu_{\tau+1}-\mu_{\tau+1}^*)\right) =0,  \label{eq:kkt1} \\
	&0\leq D_k \bv_k -N_k^\prime  \mu_k +L_k^\prime  x_k +d_k \perp \bv_k\geq 0, \label{eq:kkt2}\\
	&0\leq M_k x_k+N_k \bv_k +r_k \perp \mu_k \geq 0. \label{eq:kkt3} 
	\end{align}
	\end{subequations}
	Let $x_k^*,~k\in \mathcal K$ be the state trajectory generated by the solution $\{(\mathbf{v}_k^*,\mu_k^*),~k \in \mathcal K\}$ using \eqref{eq:State ForwardEq3}. Next, the 
	co-state defined by $\lambda_k^*=H_kx_k^*+\beta_k$ along with the state vector $x_k^*$ solve
	the two point boundary value problem \eqref{eq:lemtbvp}. Then, from Lemma \ref{lem:NCtoSC},
	it follows that  $ \bu_k^*+T_k^{-1}\mathbf{B}_k^\prime \left(E_{k+1} A_k x_k^*+e_{k+1}\right) =0$
	for all $k\in \mathcal K \backslash \{K\}$. Using this in \eqref{eq:kkt1}
	we obtain 
	\begin{align*} 
	\mathbf{B}_k^\prime \left(L_{k+1} (\bv_{k+1}-\bv_{k+1}^*)-M_{k+1}^\prime (\mu_{k+1}-\mu_{k+1}^*)
	\right)&+ \sum_{\tau=k+1}^{K-1}\left(A_{\tau}A_{\tau-1}\cdots A_{k+1} \mathbf{B}_k\right)^\prime (
	L_{\tau+1}(\bv_{\tau+1}-\bv_{\tau+1}^*)\\&-\sum_{\tau=k+1}^{K-1}  M^\prime_{\tau+1}(\mu_{\tau+1}-\mu_{\tau+1}^*)=0.
	\end{align*}
	The above equation and the remaining equations \eqref{eq:kkt2} and \eqref{eq:kkt2} are satisfied by $(\tilde{\bv}_k^*,\mu^*_k)$ as they are solutions of $\mathrm{pLCP}(x_k^*)$. 
	This implies, we have shown that the solutions are of the $\mathrm{LCP}(x_0)$ along with $\mathrm{pLCP}(x_0)$ indeed provide the optimal solution of the OCP1. Since the OCP1
	is associated with OLDPG we have that $(\tilde{\bu}^*,\tilde{\bv}^*)$, obtained from solutions of  $\mathrm{LCP}(x_0)$ and  $\mathrm{pLCP}(x_0)$, provides
	an open-loop Nash equilibrium of NZDG1.
\end{proof}
\section{Illustration : Smart grid system with energy storage}\label{sec:Examples}
To illustrate our results, we consider a smart grid system with energy storage. Smart grids provide opportunities for exploring distributed renewable energy sources. However, integrating solar and wind based sources has been a challenge in meeting the demand  due to their intermittent nature. Energy storage systems become critical in providing continuous power in the case of interruption and are often used as an emergency power supply during unforeseen outages  ( \cite{Oh:11} ). Power utilities can cut their generation costs by storing energy during the off‐peak hours and releasing during the peak hours. 
So, installing energy storage systems is crucial for an efficient, reliable and resilient smart grid, see \cite{Kolokotsa:19}. However, setting up a centralized energy storage unit for all the smart-grid users is not practical due to high set up costs and maintenance. One alternative would be to incentivize prosumers to install energy storage units at their homes; see Figure \ref{fig:SGdiagram} for an illustration. In this way, the smart grid storage system becomes decentralized. Additionally, the storage for each user is limited by the total resources available in the grid as well as the battery capacity. In \cite{zazo:16}, the authors study an energy demand problem in smart-grids without energy storage, and model the decision problem as a dynamic non-cooperative game without constraints. 
We build upon the model studied in \cite{zazo:16} by incorporating energy storage incentives for the prosumers, and model the decision problem as a dynamic game with inequality constraints.
\begin{figure}[h]
	\centering
	\includegraphics[scale = 0.375]{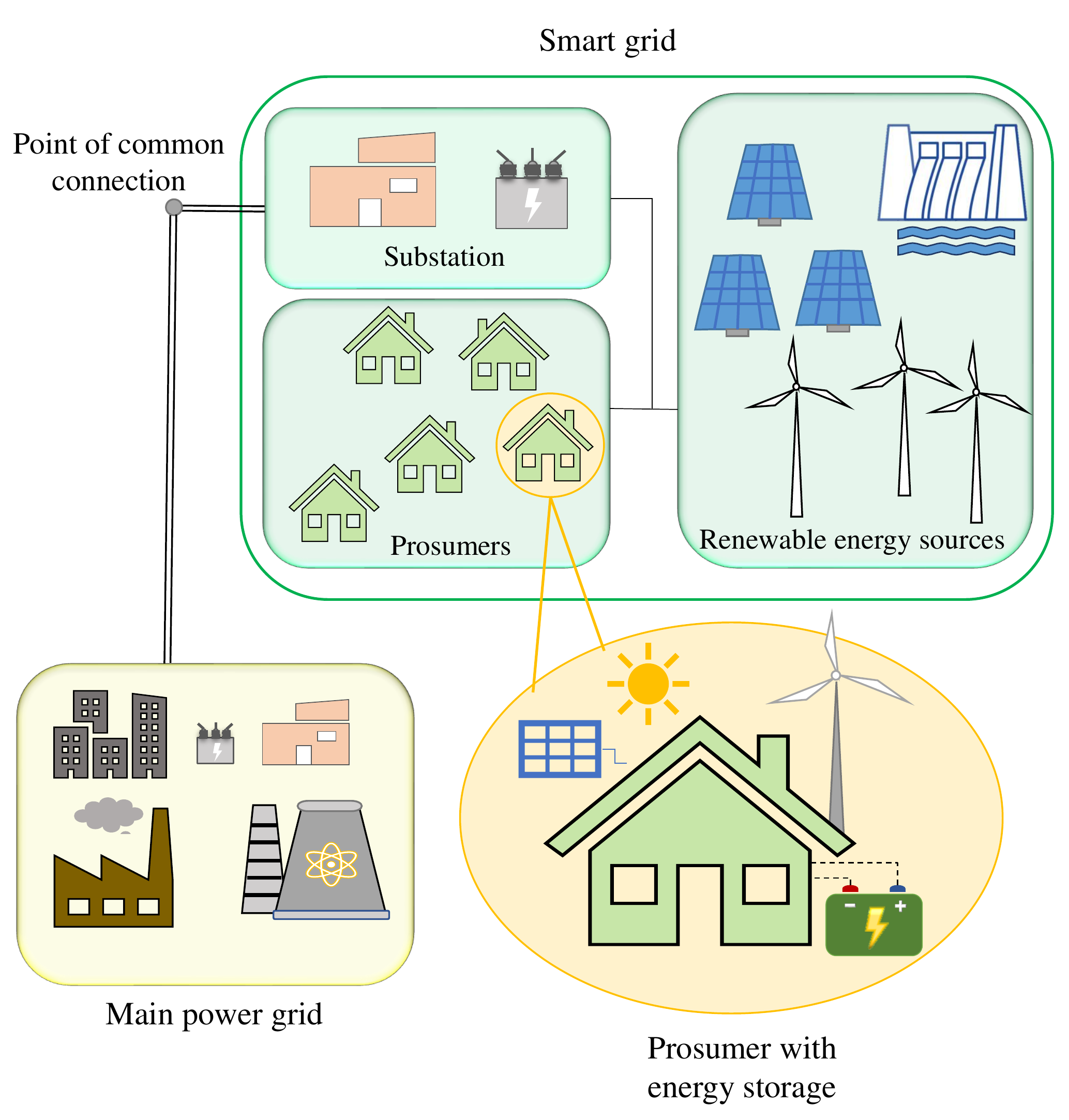}
	\caption{A smart grid system with energy storage}
	\label{fig:SGdiagram}
\end{figure} 
Consider a  smart grid with $N$ users who utilize the smart grid resources for different activities like heating, lighting, and other home appliances. Let $\mathcal{N} = \{1,2,\cdots,N\}$ be the set of users and the total time period be $\mathcal{K} = \{1,2,\cdots,K\}$. The final time $K$ is determined by the uniform interval with which the electrical data is processed in the grid in a day. For instance, if the data is processed every two hours in a day, $K =12$. The smart grid consists of $S$ type of energy resources such as solar, hydroelectric, coal etc., and each user $i \in \mathcal{N}$ consumes energy for $m_i$ activities. All resources are shared by all users. The state of the game at each time step, $X_k \in \mathbb{R}^S$ is the total amount of consumable resources in the smart grid. In the smart grid, users act as prosumers, implying that they not only consume energy, but also contribute the excess energy produced by renewable resources back to the grid. Furthermore, the resources can be autonomously recharged. Therefore, state of the game is governed by the discrete time dynamics
\begin{align}
X_{k+1} = \tilde{A}_k X_k +\sum_{i=1}^{N}\tilde{B}_k^i I_k^i,~k \in \mathcal{K} \backslash \{K\},
\end{align}
where, at time step $k$,  $\tilde{A}_k \in \mathbb{R}^{S \times S}  $ governs the energy which is  autonomously depleted or replenished,  $I_k^i \in \mathbb{R}^{m_i} $ denotes the amount of resources consumed or contributed by user $i$  and $\tilde{B}_k^i \in \mathbb{R}^{S\times m_i}$ is the weight associated with the resource expenditure or contribution. The smart grid authority has provided a battery storage unit for each user as a secondary storage unit. This is provided for the purpose of island-ed mode, operation under unforeseen disconnection of the smart grid from the main power grid; see \cite{ray:20}. The amount of resource stored in each player's battery unit at time instant $k$ is $K_k^i \geq 0 $. The total energy storage in batteries is limited to be $\epsilon_k >0$ units lower than the total resources of the grid, since excess of battery storage only results in higher storage costs. The limiting factor $\epsilon_k$ is chosen in such a way that the maximum storage limit represents the maximum energy required for the users to perform the essential activities during island-ed mode operation. Further, the amount of energy stored in the battery is limited by it's maximum capacity denoted by $K_{max}^i$. Therefore the constraints on the battery storage are given by
\begin{subequations}
	\begin{align}
	& \sum_{i \in \mathcal{N}} K_k^i  \leq  \sum_{i \in S}X_k^i - \epsilon_k, \label{eq:totalstorage}\\
	& 0 \leq K_k^i \leq K_{max}^i, ~ i \in \mathcal{N}.
	\end{align}
\end{subequations}
The costs incurred by users are attributed to unsatisfied demand, unbalanced resources and battery storage. Each user has a target demand to meet which is given by $P_k^i  X_k $, where  $P_k^i \in \mathbb{R}^{m_i \times S}$ denotes the demand matrix. The cost associated with unsatisfied demand is characterized by the following quadratic  function 
\begin{align}
(\Pi_k^i)_{ud} = \frac{1}{2}\left(P_k^i X_k -I_k^i\right)^\prime \tilde{ R}_k^i\left(P_k^i X_k -I_k^i\right),~k \in \mathcal K \backslash \{K\},
\end{align}
where $\tilde{R}_k^i ={ r}_k^i \mathbf{I} \in \mathbb{R}^{m_i \times m_i}$ with $r_k^i>0$. 
For higher values of the parameter $r_k^i$ the user prioritizes in minimizing the unsatisfied demand cost compared to other costs. Next, if the resources at a time step are higher than at the previous step, then there is a cost associated with storage. Similarly, there are also costs associated with productivity loss between consecutive time periods. These costs are modeled as 
\begin{align}
(\Pi_k^i)_{ur} = \frac{1}{2} \left( X_k -{X}_{k-1}\right)^\prime \tilde{Q}_k\left(X_k -{X}_{k-1}\right),
\end{align}
where $\tilde{Q}_k = q_k \mathbf{I} \in \mathbb{R}^{S \times S}$ with $q_k^i>0$.
For higher values of the parameter $q_k^i$ the user prioritizes in keeping the resources at  steady state levels without spikes. Each player incurs a battery storage cost which is given by 
\begin{align}
(\Pi_k^i)_{bs} = \frac{1}{2} K_k^{i^2} b_k^i,  
\end{align}
where $b_k^i$ represents the battery storage cost per energy unit for the user $i$. Along with these costs, there is an incentive provided to users for energy storage, which has two components: a player specific incentive which depends only on the player's battery resource, and a common incentive which depends on the total grid resource. The total incentive for Player $i$ is given by 
\begin{align} 
(\Pi_k^i)_{ic} = a_k^{i} K_k^i + X_k^\prime \tilde{L}_k\begin{bmatrix} {K}_k^1 & \hdots & K_k^N \end{bmatrix}^{\prime} ,
\end{align} where the parameter ${a}_k^i$ reflects the player specific incentive and the parameter $\tilde{L}_k \in \mathbb{R}^{S \times N}$ the common incentive. The salvage cost incurred by Player $i$ for the amount of resources in the last stage $K$ which is given by 
\begin{align}
\Pi^i_K = \frac{1}{2} X_K ^\prime \tilde{Q}_K  X_K,~\text{with}~\tilde{Q}_K = q_K \mathbf{I} \in \mathbb{R}^{S \times S}.
\end{align} 
Player $i$ using the consumption (or contribution) and storage schedules $\{I_k^i,~k\in \mathcal K \backslash \{K\},~K_k^i,~k\in \mathcal K\}$ seeks to minimize the total cost given by
\begin{align*} 
J^i &= \Pi_K^i+\sum_{k=0}^{K-1} \left[(\Pi_k^i)_{ud} +(\Pi_k^i)_{ur}\right] +\sum_{k=0}^K \left[(\Pi_k^i)_{bs}-(\Pi_k^i)_{ic}\right]\\
&=
\frac{1}{2}X_k ^\prime \tilde{Q}_K  X_k  + \sum_{k=0}^{K-1}\left(\frac{1}{2}\left(P^i X_k -I_k^i\right)^\prime \tilde{R}_k^i\left(P^i X_k -I_k^i\right)+\frac{1}{2}\left( X_k -{X}_{k-1}\right)^\prime \tilde{Q}_k\left( X_k -{X}_{k-1}\right)\right)\\ &+\sum_{k=0}^{K}\left(\frac{1}{2}K_k^{i^2} b_k^i - \left( a_k^iK_k^i +X_k ^\prime \tilde{L}_k^i \begin{bmatrix} {K}_k^1 & \hdots & K_k^N \end{bmatrix}^{\prime} \right)\right). 
\end{align*}
We transform the above dynamic game problem with inequality constraints to the standard form NZDG1 as follows.
\begin{align*}
& x_k  = \begin{bmatrix} X_k ^\prime &  {X}_{k-1}^\prime\end{bmatrix}^\prime, ~\mathbf{v}_k = \begin{bmatrix}
K_k^{1} & \dots & K_k^{N}
\end{bmatrix}^\prime,\\ &{u}_k^i = P^i X_k  - I_k^i,~  \mathbf{u}_k = \begin{bmatrix}
u_k^{1^\prime} & \dots & {u}_k^{N^\prime}
\end{bmatrix}^\prime.
\end{align*}
The smart grid resource allocation problem with energy storage is modeled as NZDG1 with parameters
defined as follows
\begin{align*}
&A_k \triangleq \begin{bmatrix} \tilde{A}_k+ \sum_{i = 1}^{N} \tilde{B}_k^i P^i & \mathbf{0}_{S \times S} \\ \mathbf{I}_S & \mathbf{0}_{S \times S}
\end{bmatrix}, \quad {B}_k^i \triangleq \begin{bmatrix} 
-\tilde{B}_k^i \\ \mathbf{0}_{S \times m_i}\end{bmatrix}, \\& {Q}_K =  \begin{bmatrix} 1 & 0 \\ 0 & 0 
\end{bmatrix} \otimes  \tilde{Q}_K ,\quad {Q}_k =  \begin{bmatrix} 1 & -1 \\ -1 & 1 
\end{bmatrix} \otimes \tilde{Q}_k , ~d_k^i =  -a_k^i \mathbf{e}_i, \\ &D_k^i = b_k^i (\mathbf{e}_i \mathbf{e}_i^\prime),~R_k^i = \mathbf{0}_{m_1 \times m_1}\oplus \cdots \oplus \tilde{R}_k^i \oplus \mathbf{0}_{m_N \times m_N} ,\\ 
&  M_k = \begin{bmatrix} 
\mathbf{1}_{1 \times S} & \mathbf{0}_{1 \times S} \\ \mathbf{0}_{N \times S} & \mathbf{0}_{N \times S}  
\end{bmatrix},~  N_k = -\begin{bmatrix} \mathbf{1}_{1 \times N}\\ \mathbf{I}_{N} \end{bmatrix},  \\&r_k=\begin{bmatrix} -\epsilon_k & K_{max}^1 & K_{max}^2 & \hdots & K_{max}^N \end{bmatrix}^\prime.
\end{align*}
Here, we observe that cost matrices $Q_k,~D_k^i,~d_k^i,~L_k,~k \in \mathcal{K}$ and $R_k^i,~k \in \mathcal{K} \backslash \{K\}$ satisfy the conditions in \eqref{eq:potcond1}. Therefore we can obtain the OCP1 associated with the related OLPDG by  \eqref{eq:ocpcond1} from Theorem \ref{th:LQ Restrictions}. 

For illustration purpose, we assume  $K = 12$, that is, that the data is processed every two hours in a day. We assume $S=4$ resources, $N=2$ users and $m_{i} = 2$ for both the users. The remaining parameters are taken as follows
\begin{align*}
&q_K = 2.5 ,~q_k = 1 ,~ r_k^1  = r_k^2 =  0.7,~b_k^1 =  b_k^2 = 1.6,\\ &a_k^1 = 3.4,~a_k^2 = 4,~\epsilon_k = 3.5,~ K_{max}^1 = 11.2,~ K_{max}^2 = 12.2,\\  
&\tilde{ L}_k =0.5 *\mathbf{1}_{3 \times 2},~ \tilde{A}_k = \mathbf{I}_S,~ 
 P_k^1 = P_k ^2 =0.375* \mathbf{1}_{2 \times 3},\\ 
 &\tilde{B}_k^1 =   \tilde{B}_k^2 =  \begin{bmatrix} 0.75 & 0 \\ 0 & 0.375 \\ 0  & 0.75\end{bmatrix},~
 X_0=\begin{bmatrix}4 \\ 4 \\ 4 \end{bmatrix} ,~X_{-1} = \begin{bmatrix}0 \\ 0\\ 0\end{bmatrix}.
\end{align*}
 We assume that both the users are identical in terms of energy demand. We can consider two households with similar energy requirements. However, we assume that user 2 with a 
 higher user specific incentive for storage, and reflected in the parameter values $a_k^1 = 3.4$ and $a_k^2 = 4$. Subsequently, the storage capacity of user 1's batteries are set to $K_{max}^{1} = 11.2$ units, and for user 2 are set as $K_{max}^2 = 12.2.$ units. It can be verified that the sufficient conditions in Lemma \ref{lem:NCtoSC} and Theorem \ref{thm:LCPOLNE} are satisfied with the chosen parameters. We have used the freely available software, the PATH solver (see \url{http://pages.cs.wisc.edu/~ferris/path.html}), for solving the linear complementarity problem \eqref{eq:LCP1}.

 Figure \ref{fig:alloc1} illustrates the evolution of resources in the grid. We have assumed three sources of energy in the grid. The consumption or contribution to the grid is represented in Figure \ref{fig:alloc2}. When a user consumes resources from the grid for an activity, the actual utility of the user  corresponding to this activity is negative. Likewise, contribution to the grid results in positive utility. Here,   $I_k^{i1},I_k^{i2}$ denote the consumption or contribution by the user $i$ for $m_i =2 $ activities.  We note that the consumption or contribution of both the users are identical due to the identical demand costs. Further, we observe that both the users switch from contribution to consumption for the first activity at time period  $k = 8$ and for the second activity at time period $k = 6$. As the users start to consume from the grid, the grid resources decrease which is evident from Figure \ref{fig:alloc1}. The evolution of battery storage decision for both the users is shown in Figure \ref{fig:alloc3}. User 2 has a higher user specific incentive as well as higher maximum storage  capacity in comparison with user 1. Consequently, user 2 stores higher amount of energy in its battery. User 1 utilizes the full battery storage capacity from $k =3$ to $k =11$, and user 2  utilizes full battery storage capacity from $k =4$ to $k =10$. At the final time step, $K =12$, the total grid resources fall to a lower value, due to the salvage cost. Since the total battery storage is also limited by the total grid resources and $\epsilon_k$, the battery storage of both the users also decreases at this stage.   
 
 \begin{figure*}[ht]
 	\centering 
 	\begin{subfigure}[t]{.225\textwidth}
 	\includegraphics[scale =0.5]{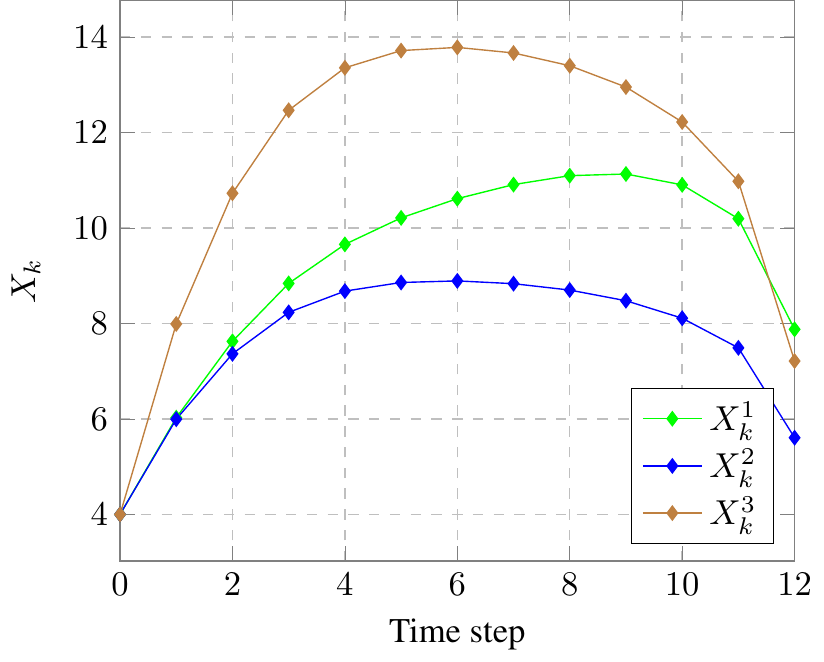}
 	\caption{Evolution of resources}
 		\label{fig:alloc1}
 	\end{subfigure} \qquad  
 	\begin{subfigure}[t]{.225\textwidth}
    \includegraphics[scale =0.5]{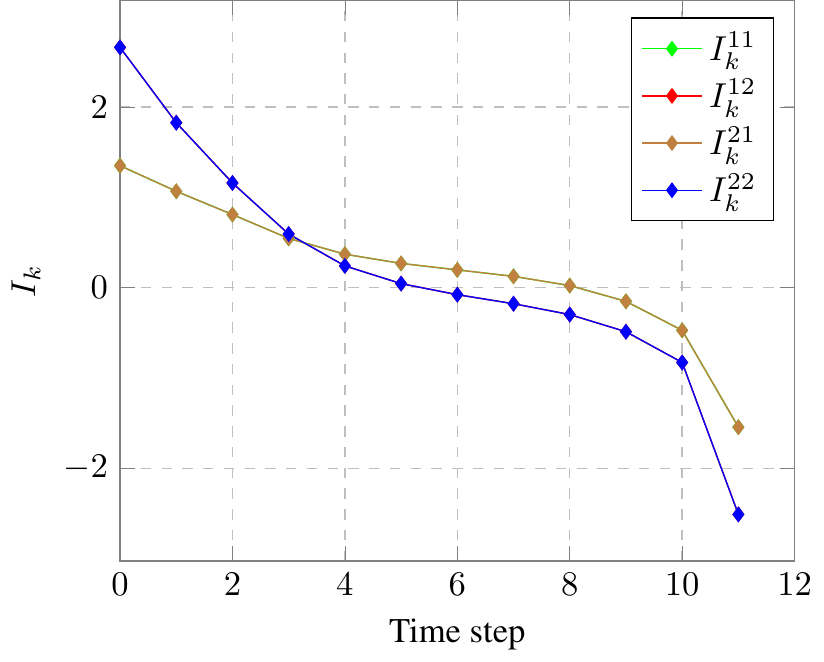}
 		\caption{Consumption or contribution}
 		\label{fig:alloc2}
 	\end{subfigure}\qquad 
 \begin{subfigure}[t]{.225\textwidth}
    \includegraphics[scale =0.5]{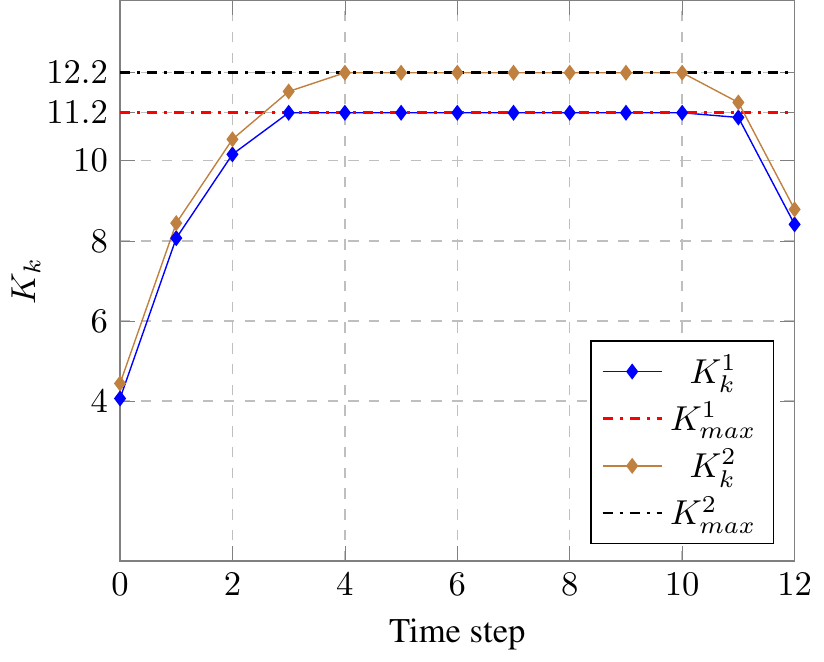}
 	\caption{Battery storage }
 	\label{fig:alloc3}
 \end{subfigure}
  \caption{State (grid resources) and decision  (consumption/contribution and battery storage) variables for the smart grid system with energy storage.}
	\label{fig:simResult} 
 \end{figure*}
\begin{figure*}[ht]
\centering 
\begin{subfigure}[t]{.225\textwidth}
   \includegraphics[scale =0.5]{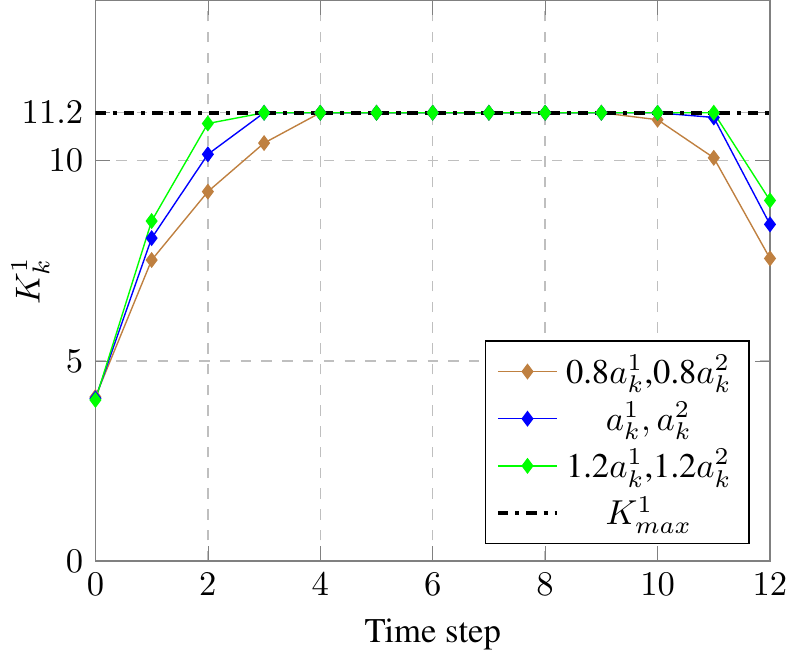}
   \caption{Battery storage: user 1 }
    \label{fig:incentive1}
\end{subfigure} \quad
\begin{subfigure}[t]{.225\textwidth}
 \includegraphics[scale =0.5]{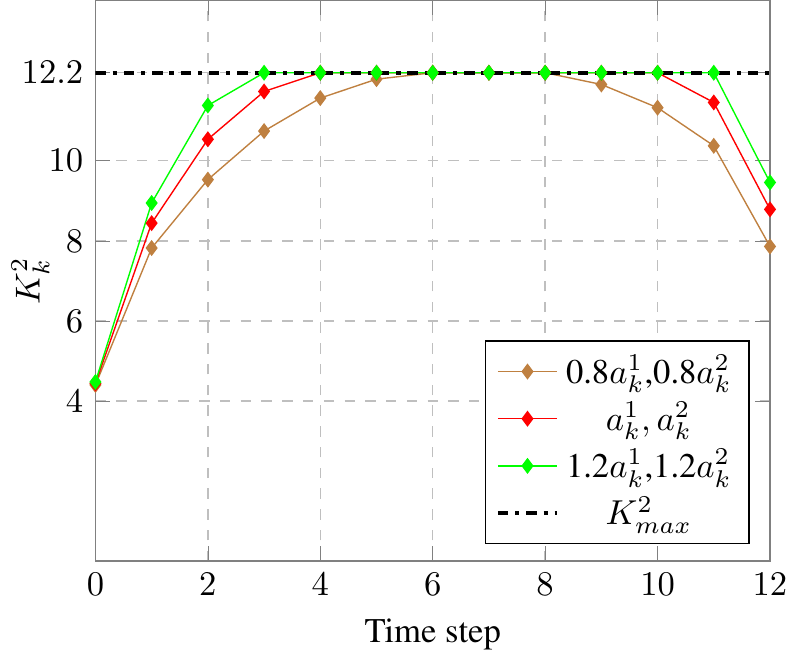}
		\caption{Battery storage: user 2 }
		\label{fig:incentive2}
\end{subfigure} \quad
\begin{subfigure}[t]{.225\textwidth}
 \includegraphics[scale =0.5]{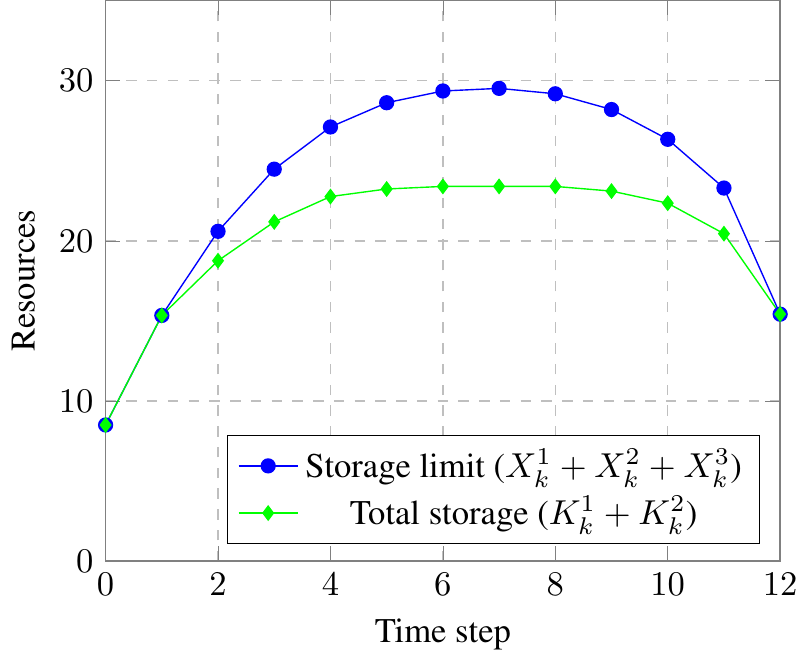}
	\caption{20\% decrease in player specific incentive}
	\label{fig:incentive3}
\end{subfigure} \quad
\begin{subfigure}[t]{.225\textwidth}
 \includegraphics[scale =0.5]{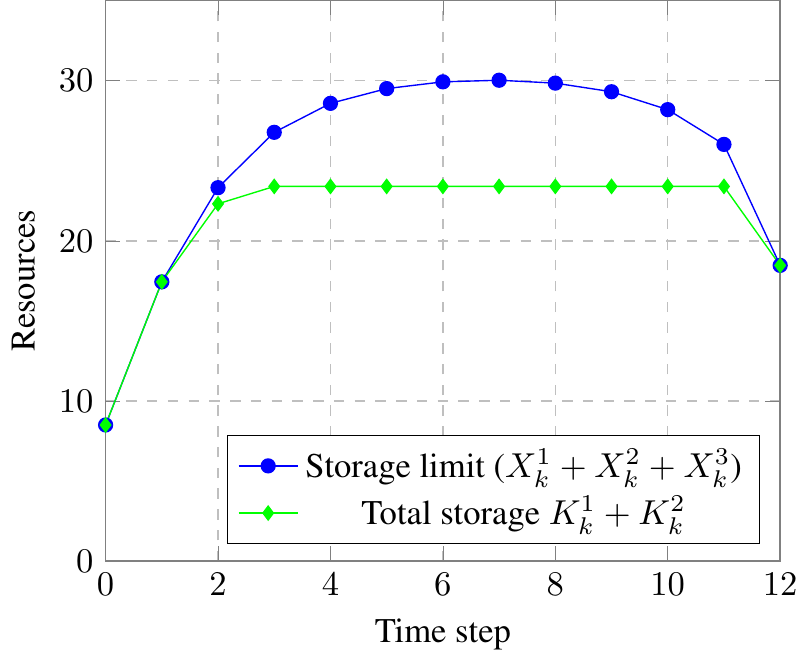}
	\caption{20\% decrease in player specific incentive}
	\label{fig:incentive4}
\end{subfigure} 
\caption{Panels (a) and (b) illustrate the variation in battery storage with change in storage incentive. Panels (c) and (d) indicate the storage limit and total battery storage with 20\% 
	variation in the incentive towards battery storage.}
	\label{fig: Storage}
\end{figure*}
Next, we analyze the effect of incentive parameter $a_k^i$ on the energy storage behavior of the users. We vary the player specific incentive parameter, $a_k^i$ for both the users. As the incentive parameter $a_k^i$ is varied with a 20\% variation around the baseline values, without varying the battery storage cost parameters, we deduce that both the users store higher amount of energy in their batteries with higher values of the incentive parameter.  Figures \ref{fig:incentive1} and \ref{fig:incentive2} illustrate that the users utilize higher capacities for a longer time when the  incentive parameter is higher. However, since the battery capacity of each player is limited, the players are unable to increase the storage continually with increase in the incentive.  Besides, full capacity is utilized for a shorter period when we lower the incentive parameter. Finally, the constraint on total storage \eqref{eq:totalstorage} also makes sure that the total energy stored in the battery is lower than the total resources by at least $\epsilon_k$. This result is illustrated in Figures \ref{fig:incentive3} and \ref{fig:incentive4}.
%

\section{Conclusions} 
 \label{sec:conclusions}
  In this paper, we studied the conditions under which a class of $N$-player non-zero sum discrete time dynamic games with inequality constraints admits a potential game structure. Drawing motivations  from the theory of static potential games, we associated an optimal control problem with inequality constraints, and derived conditions under which the solution of this optimal control problem provides a (constrained) open-loop Nash equilibrium. When the potential functions are not specified before hand, we derived conditions under which the  potential functions can be constructed using the problem data. We specialized these results to a linear quadratic setting and provided a linear complementarity problem based approach for computing the open-loop Nash equilibrium. In particular, the computed equilibrium is a refinement of the open-loop Nash equilibria obtained in \cite{Reddy:15}.  We illustrated our results with an example inspired by resource allocation in a smart grid network with energy storage. For future work, we plan to investigate the existence of potential functions for this class of games under feedback information structure.
  
\bibliographystyle{plainnat}
\bibliography{OLPDG}
\end{document}